\documentclass[draft]{amsart}
\usepackage{amsfonts}
\usepackage{amssymb}
\usepackage{amsmath}
\usepackage{amsthm}
\usepackage[dvipdfmx]{graphicx,color}
\usepackage{ascmac}
\usepackage{moreverb}
\usepackage{fancybox}
\usepackage{fancyvrb}
\usepackage{comment}

\theoremstyle{plain}
\newtheorem{theorem}{Theorem}[section]
\newtheorem{lemma}[theorem]{Lemma}
\newtheorem{cor}[theorem]{Corollary}
\newtheorem{prop}[theorem]{Proposition}

\newtheorem{question}[theorem]{Question}

\newtheorem{fact}{Fact}
\newtheorem{obs}[theorem]{Observation}

\theoremstyle{definition}
\newtheorem{definition}[theorem]{Definition}
\newtheorem{example}[theorem]{Example}

\theoremstyle{definition}
\newtheorem{remark}[theorem]{Remark}

\newcommand{\upto}{\upharpoonright}
\newcommand{\fr}{\mbox{}^\smallfrown}

\newcommand{\om}{\omega}
\newcommand{\ep}{\varepsilon}

\newcommand{\pt}[2]{#1\colon #2}

\newcommand{\ubdim}{\overline{\dim}_B}
\newcommand{\lbdim}{\underline{\dim}_B}
\newcommand{\kdim}{\dim_{\rm c.o.}}

\newcommand{\interval}{[0,1]}

\newcommand{\locex}{{\rm ex}}
\newcommand{\locexv}{{\rm ex}\mbox{-}{\rm val}}
\newcommand{\locext}{{\rm ex}\mbox{-}{\rm time}}

\newcommand{\repsp}[1]{\mathcal{ #1 }}

\title[On metric $tt$-degrees]{On a metric generalization of the $tt$-degrees and effective dimension theory}
\author{Takayuki Kihara}
\address[Takayuki Kihara]{Department of Mathematical Informatics, Graduate School of Informatics, Nagoya University, Japan}
\email{kihara@i.nagoya-u.ac.jp}
\date{}

\begin{document}
\maketitle

\begin{abstract}
In this article, we study an analogue of $tt$-reducibility for points in computable metric spaces.
We characterize the notion of the metric $tt$-degree in the context of first-level Borel isomorphism.
Then, we study this concept from the perspectives of effective topological dimension theory and of effective fractal dimension theory.
\end{abstract}

\section{Introduction}

\subsection{Summary}

In recent years, various studies have revealed that the topological/metric generalization of degrees of algorithmic unsolvability is extremely useful.
In particular, the theory of generalized Turing degrees has achieved great success.
The researchers have found a number of unexpected applications of the notion of metric/topological Turing degrees:
Day-Miller \cite{DM13} explained the behavior of Levin's neutral measures in algorithmic randomness theory;
Gregoriades-Kihara-Ng \cite{GKNta} gave a partial answer to the open problem on generalizing the Jayne-Rogers theorem in descriptive set theory;
Kihara-Pauly \cite{KP} proved a result related to descriptive set theory, infinite dimensional topology, and Banach space theory;
Andrews-Igusa-Miller-Soskova \cite{AIMS} showed that the PA degrees are first-order definable in the enumeration degrees;
Kihara-Lempp-Ng-Pauly \cite{KLNPta} established a classification theory of the enumeration degrees;
et cetera.

An analogue of {\em truth-table (tt) reducibility} for points in computable metric spaces is recently introduced by McNicholl and Rute \cite{McRu} to study uniform relativization of Schnorr randomness.
It is known that the notion of $tt$-reducibility plays an important role in effective measure theory.
Moreover, it is often claimed that Turing reducibility is a right notion for studying relative Martin-L\"of randomness, but not for relative Schnorr randomness.
Therefore, it is not surprising that the recent study of Schnorr randomness has required the metric generalization of the $tt$-degrees.
A realizability-theoretic generalization of (weak) $tt$-reducibility has also been discussed by Bauer and Yoshimura, cf.\ \cite{BaYo}.

In this article, we show that the complexity of the metric generalization of $tt$-degree is {\em one level below} the metric $T$-degree.
To be precise, Kihara-Pauly \cite{KP} has revealed that a known partial generalization of the Jayne-Rogers theorem makes the connection between the metric $T$-degree and the {\em second-level Borel isomorphism}.
In a similar manner, we apply the Jayne-Rogers theorem to characterize the notion of the generalized $tt$-degree in the context of the {\em first-level Borel isomorphism}.

This characterization is the guidepost which indicates the right way to go.
It tells us which topological concepts and techniques are relevant to study the generalized $tt$-degrees.
For instance, first-level Borel isomorphisms have appeared in several literatures in topological dimension theory, cf.\ \cite{JR79,JR79b,Sh00}.
Thus the notion of effective topological dimension naturally emerges.
In this article, we will show the following:

\begin{itemize}
\item Polish spaces $\repsp{X}$ and $\repsp{Y}$ have the same $tt$-degree structures relative to some oracle if and only if $\om\times\repsp{X}$ and $\om\times\repsp{Y}$ are first-level Borel isomorphic (Theorem \ref{thm:first-level}).
\item A $tt$-degree $\mathbf{d}$ contains a point in a computable compact metric space whose topological dimension is at most $n$ if and only if, for any $\ep>0$, there are a point $y\in\mathbb{R}^{2n+1}$ of $tt$-degree $\mathbf{d}$ and a compression algorithm $M$ such that
\[\limsup_{r\to\infty}\frac{C_{M,r}(y)}{r}<n+\ep,\]
where $C_{M,r}$ is the Kolmogorov complexity at precision $r$ w.r.t.\ the compression algorithm $M$ (Theorem \ref{thm:main-Kol-complexity}).
\item For any $n>1$, the collection of $tt$-degree structures of $n$-dimensional computable continua forms a universal countable upper semilattice (Theorem \ref{thm:usl-embedding2}).
\end{itemize}

\subsection{Preliminaries}\label{sec:preliminaries} 

We use the standard terminology in computability theory and computable analysis.
For the basics, we refer the reader to \cite{DHBook,WeiBook,SoareBook}.
A {\em computable metric space} is a triple $\mathcal{X}=(X,d,\alpha)$ such that $(X,d)$ is a separable metric space, $\{\alpha_n:n\in\om\}$ is a countable dense subset of $X$, and $(m,n)\mapsto d(\alpha_m,\alpha_n)$ is computable.
If $(X,d)$ is complete metric space, then $\mathcal{X}$ is called a {\em computable Polish space}.
By $B_\ep(x)$ we denote the open ball of center $x$ and radius $\ep$.
A rational open ball is a ball of the form $B_r(\alpha_n)$ for some $n\in\om$ and rational $r>0$.
We say that $\repsp{X}$ is {\em computably compact} if it has a c.e.\ list of finite open covers consisting of rational open balls.
A {\em computable compactum} is a computably compact computable metric space.
A set $U\subseteq \repsp{X}$ is {\em c.e.\ open} if it is a union of a computable sequence of rational open balls.
The complement of a c.e.\ open set is called a $\Pi^0_1$ set or a co-c.e.\ closed set.

A {\em Cauchy name}, or simply, a {\em name}, is an infinite sequence $p\in \om^\om$ such that $d(\alpha_{p(n)},\alpha_{p(m)})<2^{-n}$ for any $m\geq n$.
It is easy to check that this condition ensures that $\overline{B}_{2^{-n}}(\alpha_{p(n+1)})\subseteq B_{2^{-n+1}}(\alpha_{p(n)})$, where $\overline{B}_\ep(x)$ is the corresponding closed ball.
If $x$ is the limit of the sequence $(\alpha_{p(n)})_{n\in\om}$ for such $p$, then $p$ is called a (Cauchy) name of $x$.
The collection of all Cauchy names is clearly $\Pi^0_1$ in $\om^\om$.
As above, a partial Cauchy name $\sigma\in\om^{<\om}$ determines an open ball $B_{\sigma}:=B_{2^{-|\sigma|+1}}(\alpha_{\sigma(|\sigma|-1)})\subseteq\repsp{X}$.
Then define $\delta(q)$ as a unique element of $\bigcap_{n}B_{q\upto n}$ for any Cauchy name $q$, and $\delta:\subseteq\om^\om\to\repsp{X}$ is called a Cauchy representation of $\repsp{X}$.
We say that {\em $B_\sigma$ is formally included in $B_\tau$} if $d(\alpha_{\sigma(|\sigma|-1)},\alpha_{\tau(|\tau|-1)})+2^{-|\sigma|+1}<2^{-|\tau|+1}$ holds.
This is a decidable property, which implies $\overline{B}_\sigma\subseteq B_\tau$.
For instance, if $p$ is a Cauchy name, $B_{p\upto n+1}$ is formally included in $B_{p\upto n}$.
One can also define the notion of formal disjointness in a similar manner.

A partial function $f:\subseteq\repsp{X}\to\repsp{Y}$ is computable if there is a computable function $\Phi:\subseteq\om^\om\to\om^\om$ which, given a name of $x\in{\rm dom}(f)$, returns a name of $f(x)$.
A {\em computable embedding} $f\colon\repsp{X}\to\repsp{Y}$ is a computable injection which has a computable left-inverse, that is, there is a computable function $g:f[\repsp{X}]\to\repsp{X}$ such that $g(f(x))=x$ for any $x\in\repsp{X}$.

\subsection{Reducibility notions}

The metric generalization of Turing reducibility is first introduced by Miller \cite{Mil04}.

\begin{definition}[Miller \cite{Mil04}]
Let $\repsp{X}$ and $\repsp{Y}$ be computable metric spaces.
For $x\in\repsp{X}$ and $y\in\repsp{Y}$, we say that {\em $\pt{y}{\repsp{Y}}$ is $T$-reducible to $\pt{x}{\repsp{X}}$} (written $\pt{y}{\repsp{Y}}\leq_T\pt{x}{\repsp{X}}$) if there is a partial computable function $\Phi:\subseteq\om^\om\to\om^\om$ which maps every name of $x$ to a name of $y$.
\end{definition}

It is equivalent to saying that there is a partial computable function $f:\subseteq\repsp{X}\to\repsp{Y}$ such that $x\in{\rm dom}(f)$ and $f(x)=y$.
If the underlying spaces are clear from the context, we simply write $y\leq_Tx$ instead of $\pt{y}{\repsp{Y}}\leq_T\pt{x}{\repsp{X}}$.
Clearly $\leq_T$ forms a preorder, and induces an equivalence relation $\equiv_T$ on points in computable metric spaces.
In \cite{Mil04}, an $\equiv_T$-equivalence class is called a {\em continuous degree}.

Classically, $x,y\in 2^\om$, we say that {\em $y$ is truth-table reducible to $x$} (written $y\leq_{tt}x$) if there is a computable function $f:\om\to\om$ such that $y(n)=1$ iff $x$ satisfies the $f(n)$-th propositional formula.
It is known that $y\leq_{tt}x$ iff, there is a {\em total} computable function $\Phi:2^\om\to 2^\om$ such that $\Phi(x)=y$.
McNicholl-Rute \cite{McRu} introduced the notion of $tt$-reducibility in computable Polish spaces as follows:

\begin{definition}[McNicholl-Rute \cite{McRu}]\label{def:MR-tt-degree}
Let $\repsp{X}$ and $\repsp{Y}$ be computable Polish spaces.
For $x\in\repsp{X}$ and $y\in\repsp{Y}$, we say that {\em $\pt{y}{\repsp{Y}}$ is $tt$-reducible to $\pt{x}{\repsp{X}}$} (written $\pt{y}{\repsp{Y}}\leq_{tt}\pt{x}{\repsp{X}}$) if there is a total computable function $\Phi:\om^\om\to\om^\om$ which maps every name of $x$ to a name of $y$.
\end{definition}
If the underlying spaces are clear from the context, we simply write $y\leq_{tt}x$ instead of $\pt{y}{\repsp{Y}}\leq_{tt}\pt{x}{\repsp{X}}$.
Again, $\leq_{tt}$ forms a preorder, and induces an equivalence relation $\equiv_{tt}$ on points in computable Polish spaces.
Then, the $tt$-degree of $x\in\repsp{X}$ is the $tt$-equivalence class containing $x$, and written as ${\rm deg}_{tt}(\pt{x}{\repsp{X}})$ or simply ${\rm deg}_{tt}(x)$.
This gives us the {\em $tt$-degree structure $\mathcal{D}_{tt}(\repsp{X})$ of a space $\repsp{X}$}, the collection of all $tt$-degrees of points in $\repsp{X}$.

It should be careful that $\Phi$ in Definition \ref{def:MR-tt-degree} does not necessarily induce a total function $f:\repsp{X}\to\repsp{Y}$.
Instead, we have the following:

%
%
%
%

\begin{fact}[McNicholl-Rute \cite{McRu}]\label{fact:MR}
Let $x$ and $y$ be points in computable Polish spaces $\repsp{X}$ and $\repsp{Y}$, respectively.
Then, $\pt{x}{\repsp{X}}\leq_{tt}\pt{y}{\repsp{Y}}$ iff there is a partial computable function $f:\subseteq\repsp{X}\to\repsp{Y}$ such that ${\rm dom}(f)$ is $\Pi^0_1$ in $\repsp{X}$, $x\in{\rm dom}(f)$, and $f(x)=y$.
\end{fact}

If underlying spaces are computably compact, as in the classical situation, $tt$-reducibility corresponds to computability with bounded running time (cf.\ \cite[Section III.3]{OdiBook} and \cite[Section 3.8.3]{SoareBook}).

\begin{prop}\label{prop:tt-time}
Let $\repsp{X}$ and $\repsp{Y}$ be computable Polish spaces.
For a partial computable function $f:\subseteq\repsp{X}\to\repsp{Y}$, the direction (2)$\Rightarrow$(1) always holds.
Moreover, if $\repsp{X}$ is a computable compactum, the converse direction (1)$\Rightarrow$(2) holds as well, that is, the following (1) and (2) are equivalent:
\begin{enumerate}
\item There is a partial computable function $g:\subseteq\repsp{X}\to\repsp{Y}$ extending $f$ such that the domain of $g$ is $\Pi^0_1$.
\item $f$ is tracked by a Turing machine with computably bounded running time in the following sense:
There are a partial computable function $\Phi:\subseteq\om^\om\to\om^\om$ and a computable function $s:\om\to\om$ such that for every $x\in{\rm dom}(f)$,
\begin{enumerate}
\item there is an $s$-bounded name $q$ of $x$ such that $\Phi(q)$ is a name of $f(x)$, and the running time of a machine computing the $n$-th entry of $\Phi(q)$ is at most $s(n)$ for any $n\in\om$,
\item and for any $s$-bounded name $q$ of $x$, if the running time of a machine computing the $n$-th entry of $\Phi(q)$ is at most $s(n)$ for any $n\in\om$, then $\Phi(q)$ is a name of $f(x)$.
\end{enumerate}
\end{enumerate}
\end{prop}

\begin{proof}
(1)$\Rightarrow$(2):
Let $\Phi$ be a computable realizer of $g$, let $\delta:\subseteq\om^\om\to\repsp{X}$ be a Cauchy representation, and let $D$ be the domain of $g$.
Then, $\delta^{-1}[D]$ is a $\Pi^0_1$ subset of Cauchy names.
Let $T\subseteq\om^{<\om}$ be a computable tree whose infinite paths corresponds to $\delta^{-1}[D]$.
A node $\sigma\in T$ is an initial segment of a Cauchy name, and thus as explained in Section \ref{sec:preliminaries} one can naturally assign an open ball $B_\sigma$ in $\repsp{X}$ to each node $\sigma\in T$.
Since $D$ is $\Pi^0_1$ in a computable compactum $\repsp{X}$, one can effectively find $u(0)$ such that $(U_{\langle k\rangle})_{k<u(0),\langle k\rangle\in T}$ covers $D$.
In this way, one can easily find a computably bounded tree $S\subseteq T$ such that $D\subseteq\delta[S]$.
Since $[S]\subseteq[T]\subseteq{\rm dom}(\Phi)$, as in the usual argument, $\Phi\upto[S]$ can be easily simulated by a machine with computably bounded running time which ensures (a).
Since $[S]$ is $\Pi^0_1$, obviously, a machine can satisfy (b) as well.

(2)$\Rightarrow$(1):
Assume that a computable time bound $s$ is given.
Then, define
\[Q=\{q\in \om^\om:(\forall n)\;\Phi^{q\upto s(n)}_{s(n)}(n)\downarrow\mbox{ and }q(n)<s(n)\}.\]
Clearly, $Q$ is bounded $\Pi^0_1$ and $Q\subseteq{\rm dom}(\Phi)$.
Moreover the condition (a) ensures that ${\rm dom}(f)\subseteq\delta[Q]$.
We claim that there is a total computable function $\Psi:\om^\om\to\om^\om$ inducing an extension of $f$.

To see this, let $T_Q$ be a computable tree corresponding to $Q$.
Assign an open ball $B_\sigma$ to each $\sigma\in T_Q$ as above.
Assume that $q\in\om^\om$ is given, and we define $\Psi(q)$.
For $t\in\om$, wait for seeing one of the following holds: (i) there are $u\in\om$ and $\sigma\in T_Q$ of length $s(t)$ such that $B_{q\upto u}$ is formally included in $B_\sigma$; (ii) $q$ is not a Cauchy name; or (iii) $\delta(q)\not\in\delta[Q]$.
It is clear that $(i)$ and (ii) are $\Sigma^0_1$.
For (iii), note that $Q$ is computably compact; hence $\delta[Q]$ is also effectively compact and thus $\Pi^0_1$ (see also Observation \ref{obs:basic-top-dim}).
Hence, the condition (iii) is also a $\Sigma^0_1$ property.

Given $q$, if we see either (ii) or (iii), define $\Psi(q)$ to be some value just for the sake of totality of $\Psi$.
Otherwise, $q$ is a Cauchy name and $\delta(q)\in\delta[Q]$, and then $\delta(q)$ has a name $p\in Q$.
Since $\delta(q)\in B_{p\upto s(t)}$ for any $t$, there must be $u\in\om$ such that $B_{q\upto u}$ is formally included in $B_{p\upto s(t)}$ (consider $u$ such that $d(\delta(q),\repsp{X}\setminus B_{p\upto s(t)})>2^{-u}$).
Hence, we eventually find $\sigma_t\in T_Q$ of length $s(t)$ witnessing (i).
Then $\Phi(\sigma_t)\upto t+1$ determines an open set $C_t:=B_{\Phi(\sigma_t)\upto t+1}$ of diameter at most $2^{-t}$ in $\repsp{Y}$.
If $\delta(q)\in{\rm dom}(f)$, then since $\sigma_t$ extends to a name of $\delta(q)$ which ensures the premise of (b), $\Phi(\sigma_t)$ has to extend to a name of $f(\delta(q))$.
Therefore, $f(\delta(q))$ is a unique element of $\bigcap_tC_t$.
We define $\Psi(q)(t+1)=\Phi(\sigma_t)(t)$.
Note that $B_{\Psi(q)\upto t+1}$ is the result by doubling the radius of $C_t$.
Hence, if $\delta(q)\in{\rm dom}(f)$ then $(B_{\psi(q)\upto t})_{t\in\om}$ is a decreasing sequence of open balls, that is, $\Psi(q)$ is a Cauchy name representing $f(\delta(q))$.
This verifies the claim.

It remains to show that the domain of the function $g$ induced from $\Psi$ is $\Pi^0_1$.
This is the same as the standard argument to show Fact \ref{fact:MR}:
The domain of $g$ is the complement of the union of all c.e.\ open sets of the form $B_\sigma\cap B_\tau$ such that $B_{\Psi(\sigma)}$ and $B_{\Psi(\tau)}$ are formally disjoint.
\end{proof}

Note that a continuous function on a closed subset of a space $\repsp{X}$ is not necessarily extendible to a total continuous function on $\repsp{X}$.
This fact relates to the notion of topological dimension.
A topological space $Y$ is called an {\em absolute extensor of $X$} if for any closed set $P\subseteq X$, any continuous $f:P\to Y$ has a continuous extension $g:X\to Y$.
The {\em extension dimension of $X$}, written $\dim_e(X)$, is the smallest $n\in\om$ such that the $n$-sphere $\mathbb{S}^n$ is an absolute extensor of $X$ if such $n$ exists.
If there is no such $n$, $X$ is called infinite dimensional, denoted by $\dim_e(X)=\infty$.
If $X$ is separable metrizable, then the extension dimension $\dim_e(X)$ coincides with the Lebesgue covering dimension, the small inductive dimension, the large inductive dimension, etc.
Since we only deal with separable metrizable spaces in this article, we just call any of these dimensions the {\em topological dimension}, and write $\dim(X)$.

Among others, McNicholl-Rute \cite{McRu} has proven the following interesting characterization of points in computable planar arcs.

\begin{fact}[McNicholl-Rute \cite{McRu}]\label{fact:MR-computable-arc}
A point $x\in\mathbb{R}^2$ is contained in a computable arc if and only if $\pt{x}{\mathbb{R}^2}$ is $tt$-equivalent to a point in $\mathbb{R}$.
\end{fact}



%

\section{First level Borel isomorphisms}

\subsection{$tt$-equivalence and embeddability}

In this section, we show the following very useful lemma, which characterizes $tt$-equivalence in the context of computable embeddability.

\begin{lemma}\label{lem:main-lemma}
Let $\repsp{X}$ and $\repsp{Y}$ be computable Polish spaces, and $x$ and $y$ be points in $\repsp{X}$ and $\repsp{Y}$, respectively.
Then, the following are equivalent:
\begin{enumerate}
\item $\pt{x}{\repsp{X}}$ is $tt$-equivalent to $\pt{y}{\repsp{Y}}$.
\item There are a $\Pi^0_1$ set $P\subseteq\repsp{X}$ with $x\in P$ and a computable embedding $\Phi$ of $P$ into $\repsp{Y}$ such that its embedded image is $\Pi^0_1$ and $\Phi(x)=y$.
\end{enumerate}
\end{lemma}

\begin{proof}
The direction from (2) to (1) is clear.
To show the converse direction, assume that $x\in\repsp{X}$ is $tt$-equivalent to $y\in\repsp{Y}$.
By Fact \ref{fact:MR}, there are computable functions $\Phi:\repsp{X}\to\repsp{Y}$ and $\Psi:\repsp{Y}\to\repsp{X}$ with $\Pi^0_1$ domains $P\subseteq\repsp{X}$ and $Q\subseteq\repsp{Y}$, respectively, such that $\Phi(x)=y$ and $\Psi(y)=x$.
Now, we consider the following set:
\[R=\{p\in P:\Phi(p)\in Q\mbox{, and }\Psi\circ\Phi(p)=p\}.\]

We show that $R$ is the desired $\Pi^0_1$ set.
It is clear that $x\in R$ and $y\in\Phi[R]$.
Note that the condition $\Psi\circ\Phi(p)=p$ ensures that $\Phi$ is injective on $R$.
To see that $R$ is $\Pi^0_1$, note that, if $g$ is computable, the set $\{z:g(z)=z\}$ is $\Pi^0_1$ in ${\rm dom}(g)$.
The domain of $\Psi\circ\Phi$ is $P\cap\Phi^{-1}[Q]$, which is $\Pi^0_1$, and therefore we have that $R$ is $\Pi^0_1$.
Now, consider the following set:
\[S=\{q\in Q:\Psi(q)\in P\mbox{, and }\Phi\circ\Psi(q)=q\}.\]
By a similar argument as above, one can easily check that $S$ is $\Pi^0_1$.
We claim that $S=\Phi[R]$.
We first check that $\Phi[R]\subseteq S$.
Given $q\in \Phi[R]$, it can be written as $q=\Phi(p)$ for some $p\in R$.
By our definition of $R$, $q\in Q$, $\Psi(q)=\Psi\circ\Phi(p)=p\in P$, and $\Phi\circ\Psi(q)=\Phi(p)=q$.
To see $S\subseteq\Phi[R]$, let $q\in S$ be given.
Let $p$ be the value of $\Psi(q)$, which exists since $q\in Q={\rm dom}(\Psi)$.
Then, $p\in P$, $\Phi(p)=\Phi\circ\Psi(q)=q\in Q$, and $\Psi\circ\Phi(p)=\Psi(q)=p$.
This shows that $p\in R$, and therefore $q\in\Phi[R]$.

Since $\Phi$ is injective on $R$, this argument also shows that $\Phi^{-1}$ agrees with $\Psi$ on $\Phi[R]$.
Consequently, $\Phi\upto R$ is an embedding of a $\Pi^0_1$ set $R$ into $\repsp{Y}$, where $x\in R$, and its embedded image $S$ is $\Pi^0_1$ and contains $y$.
This verifies our claim.
\end{proof}

It should be noted that the proof of Lemma \ref{lem:main-lemma} does not require completeness (or even metrizability) of underlying spaces except for Fact \ref{fact:MR}.
Therefore, if we adopt the characterization in Fact \ref{fact:MR} as the definition of $tt$-reducibility, then Lemma \ref{lem:main-lemma} holds for more general represented spaces.

Fix an oracle $\alpha\in 2^\om$.
Let $\repsp{X}$ and $\repsp{Y}$ be $\alpha$-computable metric spaces.
We say that {\em $y\in\repsp{Y}$ is $tt$-reducible to $x\in\repsp{X}$ relative to $\alpha$} (written as $y\leq_{tt}^\alpha x$) if there is a partial $\alpha$-computable function $\Phi:\subseteq\repsp{X}\to\repsp{Y}$ such that ${\rm dom}(\Phi)$ is $\Pi^0_1(\alpha)$ in $\repsp{X}$, $x\in{\rm dom}(\Phi)$, and $\Phi(x)=y$.
The $\alpha$-relative $tt$-bireducibility $\equiv^\alpha_{tt}$ defines an equivalence relation on points in $\alpha$-computable metric spaces.
Then, the $\alpha$-relative $tt$-degree of $x\in\repsp{X}$ is the equivalence class containing $x$, and we define the {\em $\alpha$-relative $tt$-degree structure $\mathcal{D}_{tt}^\alpha(\repsp{X})$ of a space $\repsp{X}$} as the collection of all $\alpha$-relative $tt$-degrees of points in $\repsp{X}$.

The following lemma is a straightforward relativization of Lemma \ref{lem:main-lemma}.

\begin{lemma}\label{lem:main-embedd-relativized}
Let $\repsp{X}$ and $\repsp{Y}$ be $\alpha$-computable metric spaces for some oracle $\alpha$, and $x$ and $y$ be points in $\repsp{X}$ and $\repsp{Y}$, respectively.
Then, the  following are equivalent:
\begin{enumerate}
\item $\pt{x}{\repsp{X}}$ is $tt$-equivalent to $\pt{y}{\repsp{Y}}$ relative to some $\beta\geq_T\alpha$.
\item There are a closed set $P\subseteq\repsp{X}$ with $x\in P$ and a topological embedding $\Phi$ of $P$ into $\repsp{Y}$ such that its embedded image is closed and $\Phi(x)=y$.
\end{enumerate}
\end{lemma}

\subsection{First-level Borel isomorphism}

Let $\repsp{X}$ and $\repsp{Y}$ be topological spaces.
An $F_\sigma$ set is a countable union of closed sets.
A $\mathbf{\Delta}^0_2$ set is a $F_\sigma$ set whose complement is also $F_\sigma$.
A function $f:\repsp{X}\to\repsp{Y}$ is a {\em first-level Borel function} if $f^{-1}[A]$ is $F_\sigma$ for any $F_\sigma$ set $A\subseteq\repsp{Y}$.
A {\em first-level Borel isomorphism} between $\repsp{X}$ and $\repsp{Y}$ is a bijection $f:\repsp{X}\to\repsp{Y}$ such that both $f$ and $f^{-1}$ are first-level Borel functions.
See also \cite{JR79,JR79b,Sh00}.
In this section, we give a characterization of $tt$-reducibility in the context of a first-level Borel isomorphism.

\begin{theorem}\label{thm:first-level}
The following are equivalent for Polish spaces $\repsp{X}$ and $\repsp{Y}$:
\begin{enumerate}
\item $\om\times\repsp{X}$ and $\om\times\repsp{Y}$ are first-level Borel isomorphic.
\item $\mathcal{D}_{tt}^\alpha(\repsp{X})=\mathcal{D}_{tt}^\alpha(\repsp{Y})$ for some oracle $\alpha$.
\end{enumerate}
\end{theorem}

A function $f:\repsp{X}\to\repsp{Y}$ is said to be an {\em $F_\sigma$-map} (see \cite{JR80}) if it is a first level Borel function (i.e., $f^{-1}[A]$ is $F_\sigma$ for any $F_\sigma$ set $A\subseteq\repsp{Y}$), and moreover $f[B]$ is $F_\sigma$ for any $F_\sigma$ set $B\subseteq\repsp{X}$.
We define a $\mathbf{\Delta}^0_2$-map in a similar manner.


\begin{lemma}\label{lem:obs:delta02}
If there is an injective $F_\sigma$ map from $\repsp{X}$ into $\om\times\repsp{Y}$, then there is also an injective $\mathbf{\Delta}^0_2$-map from $\repsp{X}$ into $\om\times\repsp{Y}$.
\end{lemma}

\begin{proof}
Let $f:\repsp{X}\to\repsp{Y}$ be a injective $F_\sigma$ map.
Then $f[\repsp{X}]$ is $F_\sigma$ in $\repsp{Y}$; hence it is written as a union of pairwise disjoint $\mathbf{\Delta}^0_2$ subsets of $\repsp{Y}$, say $f[\repsp{X}]=\bigcup_nP_n$.
Define $g(x)=(n,f(x))$, where $n$ is the unique number such that $f(x)\in P_n$.
It is clear that $g:\repsp{X}\to\om\times\repsp{Y}$ is injective.
It remains to show that $g$ is a desired $\mathbf{\Delta}^0_2$-map.

For the image, let $A$ be a $\mathbf{\Delta}^0_2$ set in $\repsp{X}$.
Since $f$ is injective and $P_n\subseteq f[\repsp{X}]$, we have $P_n\setminus f[\repsp{X}\setminus A]=P_n\setminus (f[\repsp{X}]\setminus f[A])=P_n\cap f[A]$.
Since both $A$ and $\repsp{X}\setminus A$ are $F_\sigma$, and $f$ is a $F_\sigma$-map, we have that $P_n\cap f[A]$ is $F_\sigma$ and $P_n\setminus f[\repsp{X}\setminus A]$ is $G_\delta$; hence $P_n\cap f[A]$ is $\mathbf{\Delta}^0_2$.
Now it is clear that $(n,y)\in g[A]$ iff $y\in P_n\cap f[A]$, which is a $\mathbf{\Delta}^0_2$ condition in $\om\times\repsp{Y}$.
Hence, $g[A]$ is $\mathbf{\Delta}^0_2$ whenever $A$ is $\mathbf{\Delta}^0_2$.

For the preimage, let $B$ be a $\mathbf{\Delta}^0_2$ set in $\om\times\repsp{Y}$, and define $B_n=\{y:(n,y)\in B\mbox{ and }y\in P_n\}$, which is clearly $\mathbf{\Delta}^0_2$ in $\repsp{Y}$.
Moreover, It is easy to see that $g(x)\in B$ iff $f(x)\in B_n$ for the unique $n$ such that $f(x)\in P_n$.
Hence, we have
\[x\in g^{-1}[B]\iff (\exists n)\;x\in f^{-1}[B_n]\iff (\forall n)\;[x\in f^{-1}[P_n]\;\rightarrow\;x\in f^{-1}[B_n]].\]
Since $f$ is a $F_\sigma$ map, $f^{-1}[B_n]$ and $f^{-1}[P_n]$ are $\mathbf{\Delta}^0_2$, and therefore, the above condition gives a $\mathbf{\Delta}^0_2$ definition of $g^{-1}[B]$ as desired.
\end{proof}


Our main idea for proving Theorem \ref{thm:first-level} is the use of the Cantor-Schr\"oder-Bernstein (CSB) argument.
Note that there is no analogue of the CSB theorem in the category of topological spaces and continuous functions; fortunately however we have the following variant of the CSB argument for first-level Borel functions.

\begin{lemma}[see also {\cite[Lemma 5.2]{JR79b}}]\label{lemma:CSB}
Assume that there are injective $F_\sigma$-maps $f:\repsp{X}\to\repsp{Y}$ and $g:\repsp{Y}\to\repsp{X}$.
Then, there is a first level Borel isomorphism between $\om\times\repsp{X}$ and $\om\times\repsp{Y}$.
\end{lemma}

\begin{proof}
By Lemma \ref{lem:obs:delta02}, we can assume that there are injective $\mathbf{\Delta}^0_2$-maps $f:\repsp{X}\to\repsp{Y}$ and $g:\repsp{Y}\to\repsp{X}$ (by replacing $\repsp{X}$ and $\repsp{Y}$ with $\om\times\repsp{X}$ and $\om\times\repsp{Y}$, respectively, if necessary).
We inductively define $\mathbf{\Delta}^0_2$ sets $\repsp{X}_n\subseteq\repsp{X}$ and $\repsp{Y}_n\subseteq\repsp{Y}$ for $n\in\om$.
Begin with $\repsp{X}_0=\emptyset$.
Assume that $\repsp{X}_n$ is given.
Then, we define 
\[
\repsp{Y}_n=f[\repsp{X}\setminus\repsp{X}_n],\qquad \repsp{X}_{n+1}=g[\repsp{Y}\setminus\repsp{Y}_n].
\]

By induction, we see that $\repsp{X}_n\subseteq\repsp{X}$ and $\repsp{Y}_n\subseteq\repsp{Y}$ are $\mathbf{\Delta}^0_2$ since $f$ and $g$ are $\mathbf{\Delta}^0_2$-maps.
We now define $h:\om\times\repsp{X}\to\om\times\repsp{Y}$ as follows:
\[
h(n,x)=
\begin{cases}
(n-1,g^{-1}(x)) & \mbox{ if }x\in\repsp{X}_n,\\
(n,f(x)) & \mbox{ if }x\not\in\repsp{X}_n.
\end{cases}
\]

As in the usual CSB argument, one can check that $h$ is bijective, and
\[
h^{-1}(n,x)=
\begin{cases}
(n,f^{-1}(x)) & \mbox{ if }x\in\repsp{Y}_n,\\
(n+1,g(x)) & \mbox{ if }x\not\in\repsp{Y}_n.
\end{cases}
\]

This shows that both $h$ and $h^{-1}$ are $\mathbf{\Delta}^0_2$-maps.
Consequently, $h$ is a first-level Borel isomorphism between $\om\times\repsp{X}$ and $\om\times\repsp{Y}$.
\end{proof}

For a pointclass $\Gamma$, we say that a function $f:\repsp{X}\to\repsp{Y}$ is {\em $\Gamma$-piecewise continuous} if there is a $\Gamma$-cover $(G_i)_{i\in\om}$ of $\repsp{X}$ such that $f\upto G_i$ is continuous for each $i\in\om$.
To show Theorem \ref{thm:first-level}, we will use the following fact:

\begin{fact}[Jayne-Rogers \cite{JR82}]\label{thm:Jayne-Rogers}
Let $\repsp{X}$ be an analytic subset of a Polish space, and $\repsp{Y}$ be a separable metrizable space.
The following are equivalent for $f:\repsp{X}\to\repsp{Y}$:
\begin{enumerate}
\item $f$ is a first-level Borel function.
\item $f$ is closed-piecewise continuous.
\end{enumerate}
\end{fact}

\begin{remark}
Indeed, for Polish spaces $\repsp{X}$ and $\repsp{Y}$, the conditions (1)--(5) are all equivalent for $f:\repsp{X}\to\repsp{Y}$:
\begin{enumerate}
\item[(3)] $f$ is $F_\sigma$-piecewise continuous.
\item[(4)] $f$ is $G_\delta$-measurable, that is, $f^{-1}[A]$ is $G_\delta$ for any open set $A\subseteq\repsp{Y}$.
\item[(5)] $f$ is $\mathbf{\Delta}^0_2$-measurable, that is, $f^{-1}[A]$ is $\mathbf{\Delta}^0_2$ for any open set $A\subseteq\repsp{Y}$.
\end{enumerate}

If $\repsp{Y}=\mathbb{R}$, the following conditions (6) and (7) are also equivalent to (1)--(5):
\begin{enumerate}
\item[(6)] $f$ is a discrete-Baire-one (a.k.a.\ stable-Baire-one) function, that is, $f$ is the discrete limit of a sequence $(f_n)_{n\in\om}$ of continuous functions (see \cite{CL75}).
\item[(7)] $f$ is a Baire-one-star function, that is, for every nonempty closed set $C\subseteq\repsp{X}$, there is an open set $U\subseteq\repsp{X}$ such that $U\cap C$ is nonempty and $f\upto C$ is continuous on $U$ (see \cite{OM77}).
\end{enumerate}

We can think of each of the above characterizations as saying that $f$ is {\em continuous with finite mind-changes}.
For further equivalences, see also Banakh-Bokalo \cite{TB10}.
\end{remark}

A {\em $\Gamma$-piecewise embedding} is an injection $f$ such that both $f$ and $f^{-1}\upto{\rm Im}(f)$ are $\Gamma$-piecewise continuous.
Such a function is called a {\em $\Gamma$-piecewise $\Lambda$-embedding of $\repsp{X}$ into $\repsp{Y}$} if its image is $\Lambda$ in $\repsp{Y}$.
By the Jayne-Rogers theorem (Fact \ref{thm:Jayne-Rogers}) and Lemma \ref{lem:obs:delta02}, there is an injective $F_\sigma$-map from $\repsp{X}$ into $\om\times\repsp{Y}$ if and only if there is a $\mathbf{\Delta}^0_2$-piecewise $\mathbf{\Delta}^0_2$-embedding of $\repsp{X}$ into $\om\times\repsp{Y}$.

\begin{proof}[Proof of Theorem \ref{thm:first-level}]
By the above discussion and by Lemma \ref{lemma:CSB}, it suffices to show that there is a $\mathbf{\Delta}^0_2$-piecewise $\mathbf{\Delta}^0_2$-embedding of $\repsp{X}$ into $\om\times\repsp{Y}$ if and only if $\mathcal{D}_{tt}^\alpha(\repsp{X})\subseteq\mathcal{D}_{tt}^\alpha(\repsp{Y})$ for some oracle $\alpha$.

Let $f$ be a $\mathbf{\Delta}^0_2$-piecewise $\mathbf{\Delta}^0_2$-embedding from $\repsp{X}$ into $\om\times\repsp{Y}$.
Note that $f$ is clearly a closed-piecewise $\mathbf{\Delta}^0_2$-embedding.
Then, there are an oracle $\alpha\in 2^\om$, and a uniform $\Pi^0_1(\alpha)$ collections $(J_i)_{i\in\om}$, $(P_i)_{i\in\om}$, and $(Q_i)_{i\in\om}$ such that ${\rm Im}(f)=\bigcup_{i\in\om}J_i$, and that $f\upto P_i$ and $f^{-1}\upto {\rm Im}(f)\cap Q_i$ are $\alpha$-computable for any $i\in \om$.
Fix $x\in\repsp{X}$.
Then, there is $i$ such that $x\in P_i$.
Since $f\upto P_i$ is an $\alpha$-computable function with the $\Pi^0_1(\alpha)$ domain $P_i$, we have $f(x)\leq_{tt}^\alpha x$.
To see $x\leq_{tt}^\alpha f(x)$, choose $j$ such that $f(x)\in Q_j$.
Since $f(x)\in{\rm Im}(f)$, there is $k$ such that $f(x)\in J_k\cap Q_j$.
Since $f^{-1}\upto {\rm Im}(f)\cap Q_j$ is $\alpha$-computable, so is $f^{-1}\upto J_k\cap Q_j$.
Moreover, $f^{-1}\upto J_k\cap Q_j$ has the $\Pi^0_1(\alpha)$ domain $J_k\cap Q_j$, and therefore $f^{-1}(f(x))=x\leq_{tt}^\alpha f(x)$.
Consequently, we have $f(x)\equiv_{tt}^\alpha x$ for any $x\in\repsp{X}$.
This shows that $\mathcal{D}_{tt}^\alpha(\repsp{X})\subseteq\mathcal{D}_{tt}^\alpha(\om\times\repsp{Y})=\mathcal{D}_{tt}^\alpha(\repsp{Y})$.

Conversely, assume that $\mathcal{D}_{tt}^\alpha(\repsp{X})\subseteq\mathcal{D}_{tt}^\alpha(\repsp{Y})$.
Assume that the symbol $\Phi$ ($\Psi$ resp.)\ ranges over partial $\alpha$-computable functions from $\repsp{X}$ to $\repsp{Y}$ ($\repsp{Y}$ to $\repsp{X}$, resp.)\ and the symbol $P$ ($Q$ resp.)\ ranges over $\Pi^0_1(\alpha)$ subsets of $\repsp{X}$ ($\repsp{Y}$, resp.)
Define $I^\repsp{X}$ as the set of pairs (of indices of) $\langle \Phi,P\rangle$ such that the domain of $\Phi$ is $P$.
Define $I^\repsp{Y}$ in a similar manner.
By our assumption, for any $x\in\repsp{X}$, there is $y\in\repsp{Y}$ such that $x\equiv_{tt}^\alpha y$.
This is equivalent to saying that for any $x\in\repsp{X}$, there are $\langle \Phi,P,\Psi,Q\rangle\in I^\repsp{X}\times I^\repsp{Y}$ such that $\Psi\circ\Phi(x)=x$.
Now, we consider the following closed set for $e=\langle\Phi,P,\Psi,Q\rangle$:
\[R_{e}=\{x\in P:\Phi(x)\in Q\mbox{, and }\Psi\circ\Phi(x)=x\}.\]

In Lemma \ref{lem:main-lemma}, we have shown that $\Phi\upto R_e$ is an embedding of $R_e$ into $\repsp{Y}$ such that the image $\Phi[R_e]$ is closed.
Hereafter, we write the quadruple coded by $e$ as $\langle\Phi_e,P_e,\Psi_e,Q_e\rangle$.
Now, given $x\in\repsp{X}$, let $e(x)$ be the least number $e\in I^\repsp{X}\times I^\repsp{Y}$ such that $x\in R_{e}$.
Then define $R^\ast_{e}$ be the set of all $x\in\repsp{X}$ such that $e(x)=e$, that is, $R^\ast_e=R_e\setminus\bigcup_{d<e}R_d$.
Clearly, $R^\ast_{e}$ is $\mathbf{\Delta}^0_2$, and so is $\Phi_e[R^\ast_{e}]=\Psi_e^{-1}[R^\ast_{e}]$.
We now consider the following function:
\[f(x)=( e(x),\Phi_{e(x)}(x)).\]

We claim that $f$ is a $\mathbf{\Delta}^0_2$-piecewise $\mathbf{\Delta}^0_2$-embedding of $\repsp{X}$ into $\om\times\repsp{Y}$.
Note that the $e$-th section of ${\rm Im}(f)$ is $\Phi[R^\ast_e]$.
Therefore, ${\rm Im}(f)$ is $\mathbf{\Delta}^0_2$.
Clearly, $f$ is $\mathbf{\Delta^0_2}$-piecewise continuous since $(R^\ast_{e})$ covers $\repsp{X}$ and $f\upto R^\ast_{e}$ sends $x$ to $(e,\Phi_e(x))$.
To see that $f^{-1}$ is $\mathbf{\Delta^0_2}$-piecewise continuous, consider $S_{e}=\{e\}\times\Phi_e[R^\ast_{e}]$.
Then $S_{e}$ is $\mathbf{\Delta}^0_2$, and note that $(S_{e})$ covers ${\rm Im}(f)$.
Moreover, $f^{-1}\upto S_{e}$ sends $(e,y)$ to $\Psi_e(y)$.
Hence, $f^{-1}$ is $\mathbf{\Delta^0_2}$-piecewise continuous.
This verifies our claim.
\end{proof}

\begin{cor}\label{cor:dimension-is-invariant}
Let $\repsp{X}$ and $\repsp{Y}$ be Polish spaces.
If $\dim(\repsp{X})<\dim(\repsp{Y})$ then $\mathcal{D}_{tt}^\alpha(\repsp{Y})\not\subseteq\mathcal{D}_{tt}^\alpha(\repsp{X})$ for any oracle $\alpha$.
\end{cor}

\begin{proof}
It follows from the fact that the topological dimension is first-level Borel invariant (cf.\ \cite[Theorem 4.2]{JR79}) and Theorem \ref{thm:first-level}.
Here we give a more direct proof.
If $\mathcal{D}^\alpha(\repsp{Y})\subseteq\mathcal{D}^\alpha(\repsp{X})$ for some oracle $\alpha$, then every $y\in\repsp{Y}$ is $tt$-reducible to some $x\in\repsp{X}$ relative to $\alpha$.
By relativizing Lemma \ref{lem:main-lemma}, we get a $\Pi^0_1(\alpha)$ set $Q_y\subseteq\repsp{Y}$ containing $y$ which embeds into $\repsp{X}$ as a closed set.
Since $Q_y$ is homeomorphic to a subset of $\repsp{X}$, we have $\dim(Q_y)\leq\dim(\repsp{X})$ (cf.\ \cite[Theorem III.1]{HWBook}).
However, there are only countably many $\Pi^0_1(\alpha)$ set, and therefore $\repsp{Y}$ is a countable union of closed subsets whose dimension is less than or equal to $\dim(\repsp{X})$.
By the sum theorem (cf.\ \cite[Theorem III.2]{HWBook}), we conclude $\dim(\repsp{Y})\leq\dim(\repsp{X})$.
\end{proof}

\section{Effective topological dimension theory}

There are several works on first-level Borel isomorphisms in the context of topological dimension theory, cf.\ \cite{JR79,JR79b,Sh00}.
For instance, Jayne-Rogers \cite[Theorem 4.2]{JR79} showed that the topological dimension is a first-level Borel invariant.
In Theorem \ref{thm:first-level}, we characterized the notion of metric $tt$-degree in the context of first-level Borel isomorphism.
Therefore, it is natural to expect that effective topological dimension theory is useful to investigate metric $tt$-degrees.
Unfortunately, there is very few previous works on effective topological dimension theory, with the exception of \cite{Ken15}.
As a consequence, we need to develop effective topological dimension theory from the very basic.

All the results in this section are straightforward effectivizations of known topological facts, which do not involve any nontrivial computability-theoretic ideas; hence, from the computability-theoretic perspective, there is nothing interesting in our proofs in this section.


\subsection{Basic observations}

Recall that a computable compactum is a computably compact computable metric space.
An {\em $\ep$-mapping} is a function $g:\repsp{X}\to\repsp{Y}$ such that for any $y\in\repsp{Y}$ the diameter of $g^{-1}\{y\}$ is less than $\ep$.
It is easy to see the following.

\begin{obs}\label{obs:basic-top-dim}
Let $X,Y$ be computable compacta.
\begin{enumerate}
\item A computable image of a computably compactum is computably compact.
\item There is a computable transformation between a name of $A\subseteq X$ as a computable compact set and a name of $A$ as a $\Pi^0_1$ set.
\item If $f:X\to Y$ is an injective computable function, then $X$ is computably homeomorphic to $f[X]$.
\item If $f:X\to Y$ is a computable $\ep$-mapping for any $\ep>0$ then $f$ is a computable embedding.
\end{enumerate}
\end{obs}

\begin{proof}
(1) $(U_i)_{i<n}$ is a cover of $f[X]$ iff $(f^{-1}[U_i])_{i<n}$ is a cover of $X$.

(2) It is clear that every $\Pi^0_1$ subset of a computably compactum is computably compact.
Conversely, let $A\subseteq X$ be computably compact.
For any finite tuple $(U_0,\dots,U_n,V)$ of basic open sets, if $(U_i)_{i\leq n}$ covers $A$ and $U_i$ and $V$ are formally disjoint, then enumerate $V$.
Let $(V_n)_{n\in\om}$ be an enumeration of all such $V$.
Then, we get $A=X\setminus\bigcup_nV_n$.
(This is because for any $x\not\in A$, $x$ and $A$ are separated by formally disjoint open sets.
For instance, consider an open neighborhood of $x$ of rational radius $<d(x,A)/2$.)

(3) Let $g:f[X]\to X$ be the left inverse of $f$.
We show that $g$ is computable.
Given a c.e.\ open set $A\subseteq X$, since $X\setminus A$ is computably compact, so is $g[X\setminus A]$ by (1).
Therefore, by (2) we get a name of $f[X\setminus A]$ as a $\Pi^0_1$ set.
Note that $g^{-1}[X\setminus A]=f[X\setminus A]$.
Hence, we obtain a name of $g^{-1}[A]=f[X]\setminus g^{-1}[X\setminus A]$ as a c.e.\ open set in $f[X]$.

(4) If $f:X\to Y$ is a computable $\ep$-map for any $\ep>0$ then clearly $f$ is injective.
\end{proof}

A computable $T_0$-space $\repsp{X}$ is {\em computably normal} if given negative informations of disjoint closed sets $A,B\subseteq\repsp{X}$, one can effectively find positive informations of disjoint open sets $U,V\subseteq\repsp{X}$ such that $A\subseteq U$ and $B\subseteq V$ (cf.\ Grubba-Schr\"oder-Weihrauch \cite[Definition 4.3]{GSW07}).
We need a characterization of a normal space in the context of a shrinking.
A {\em shrinking} of a cover $\mathcal{U}$ of a space $\repsp{X}$ is a cover $\mathcal{V}=\{V(U):U\subseteq\mathcal{U}\}$ of $\repsp{X}$ such that $V(U)\subseteq U$.
If $\mathcal{V}$ consists of open (closed, resp.)\ sets, we say that $\mathcal{V}$ is an open (closed, resp.)\ shrinking.

\begin{lemma}\label{lem:comp-normal}
A space $\repsp{X}$ is computably normal if and only if given a finite open cover $\mathcal{U}$ of $\repsp{X}$, one can effectively find a closed shrinking $\mathcal{F}$ of $\mathcal{U}$ and an open shrinking $\mathcal{V}$ of $\mathcal{F}$.
\end{lemma}

\begin{proof}
Assume $\mathcal{U}=(U_k)_{k<\ell}$.
First put $U=U_0$, and $V=\bigcup\{W\in\mathcal{U}:W\not=U\}$.
It is straightforward to see that a space  $\repsp{X}$ is computably normal if and only if given open cover $(U,V)$ of $\repsp{X}$, one can effectively find a closed shrinking $(U^-,V^-)$.
Note that $U^-$ and $\repsp{X}\setminus U$ is a disjoint pair of closed sets.
Thus, again, by using computable normality, one can effectively find a disjoint pair $V(U),G(U)$ of open sets such that $U^-\subseteq V(U)$ and $\repsp{X}\setminus U\subseteq G(U)$.
Put $F(U)=\repsp{X}\setminus G(U)$.
Note that $V(U)\subseteq F(U)\subseteq U$, and $\mathcal{U}_1=(\mathcal{U}\setminus\{U\})\cup\{V(U)\}$ forms an open shrinking of $\mathcal{U}$.
Then proceed similar procedure with $U_1\in\mathcal{U}_1$ to get $V(U_1)\subseteq F(U_1)\subseteq U_1$. 
By iterating this procedure, we get desired $\mathcal{V}$ and $\mathcal{F}$.
\end{proof}

Let $\mathcal{U}$ be an open cover of $X$.
An {\em order} of $\mathcal{U}$ (denoted by ${\rm ord}(\mathcal{U})$) is the maximal cardinality of $\mathcal{V}\subseteq\mathcal{U}$ such that $\bigcap\mathcal{V}\not=\emptyset$.
It is well-known that $\dim(X)\leq n$ iff for any open cover $\mathcal{U}$ of $X$ there exists an open refinement $\mathcal{V}$ of $\mathcal{U}$ with ${\rm ord}(\mathcal{V})\leq n$ (cf.\ van Mill \cite[Theorem 3.2.5]{vMBook}).
We need to effectivize this fundamental dimension-theoretic fact.

\begin{lemma}\label{lem:dim-refinement}
Let $X$ be a $\Pi^0_1$ subset of $[0,1]^\om$.
Then, $\dim(X)\leq n$ iff given an open cover $\mathcal{U}$ of $X$ one can effectively find an open refinement $\mathcal{V}$ of $\mathcal{U}$ with ${\rm ord}(\mathcal{V})\leq n$.
\end{lemma}

\begin{proof}
Assume that $\dim(X)\leq n$, and let $\mathcal{U}=(U_k)_{k<\ell}$ be an open cover of $X$.
By compactness, each $U_k$ can be a finite union of basic open sets, that is, $U_k$ is of the form $X\cap\bigcup_{m<b} B^k_m$ for a finite collection $(B^k_m)_{m<b}$ of rational open balls in $\interval^\om$.
Note also that the predicate $X\cap\overline{B_0}\cap\dots\cap\overline{B_{n+1}}=\emptyset$ is c.e., where $B_m$ ranges over all finite unions of basic open sets in $\interval^\om$.
Thus, ${\rm ord}(\overline{\mathcal{V}})\leq n$, where $\overline{\mathcal{V}}=(\overline{V_k})$, is a c.e.\ predicate uniformly in a sequence $\mathcal{V}$ of basic open sets in $X$.
By normality, there must exist an open refinement $\mathcal{V}$ of $\mathcal{U}$ such that ${\rm ord}(\overline{\mathcal{V}})\leq n$.
We only need to search such $\mathcal{V}$.
\end{proof}

The following is an easy effectivization of a very basic observation (cf.\ Engelking \cite[Theorem 1.10.2]{EngBook}).

\begin{obs}\label{obs:general-position}
Given $\ep>0$ and points $q_1,\dots,q_k\in \mathbb{R}^m$, one can effectively find $p_1,\dots p_k\in \mathbb{R}^m$ in
general position such that $d(p_i,q_i)<\ep$ for any $i\leq k$.
\end{obs}

A {\em polyhedron} is a geometric realization $|\mathcal{K}|$ of a simplicial complex $\mathcal{K}$ in a Euclidean space.
We approximate a given space by a polyhedron as follows:
Let $\mathcal{U}=(U_i)_{i<k}$ be a finite open cover of $X$.
The {\em nerve of $\mathcal{U}$} is a simplicial complex $\mathcal{N}(\mathcal{U})$ with $k$ many vertices $\{p_i\}_{i<k}$ such that an $m$-simplex $\{p_{j_0},\dots,p_{j_{m+1}}\}$ belongs to $\mathcal{N}(\mathcal{U})$ iff $U_{j_0}\cap\dots\cap U_{j_{m+1}}$ is nonempty.
We define the function $\kappa:X\to |\mathcal{N}(\mathcal{U})|$ as follows:
\[\kappa(x)=\frac{\sum_{i=0}^{k-1}d(x,X\setminus U_i)p_i}{\sum_{j=0}^{k-1}d(x,X\setminus U_j)}.\]
The function $\kappa$ is called {\em the $\kappa$-mapping (or Kuratowski mapping) determined by $\mathcal{U}$ and $(p_i)_{i<k}$}.
For basics on the $\kappa$-mapping, see also Engelking \cite[Definition 1.10.15]{EngBook}, van Mill \cite[Section 2.3]{vMBook}, and Nagata \cite[Section IV.5]{NagBook}.

\subsection{The universal N\"obeling spaces}

In this section, we effectivize the imbedding theorem.
The following is easy effectivizations of \cite[Lemmas 1.11.2 and 1.11.3]{EngBook}.

\begin{obs}\label{obs:func-space-open}
Let $X$ and $Y$ be computable compacta.
\begin{enumerate}
\item If $\ep>0$ is rational, the set of all $\ep$-mappings is c.e.\ open in $C(X,Y)$.
\item If $A$ is a co-c.e.\ closed subset of $Y$, then $\{f\in C(X,Y):f[X]\cap A=\emptyset\}$ is c.e.\ open in $C(X,Y)$.
\end{enumerate}
\end{obs}

\begin{proof}
(1) Obvious.
(2) By Observation \ref{obs:basic-top-dim}, one can find an index of the computable compact set $f[X]$.
The condition $f[X]\cap A=\emptyset$ is equivalent to that $f[X]\subseteq Y\setminus A$, which is a c.e.\ condition since $f[X]$ is computably compact, and $Y\setminus A$ is c.e.\ open.
\end{proof}

A N\"obeling space $N^m_n$ be the set of all $m$-tuples $(x_\ell)_{\ell<m}$ of reals such that the number of $\ell<m$ such that $x_\ell$ is rational is at most $n$.
It is easy to effectivize the N\"obeling imbedding theorem as follows:

\begin{prop}
Every $n$-dimensional computable compactum is effectively embedded into the $n$-dimensional N\"obeling space $N^{2n+1}_n$.
\end{prop}

\begin{proof}
Given a compactum $X$, let $C_\ep(X,Y)$ be the set of all $\ep$-mapping of $X$ to $Y$.
Assume that $\dim(X)\leq n$ and $L$ is an $n$-dimensional linear subspace of $\mathbb{R}^{2n+1}$.
Then, it is known that $D_\ep=\{f\in C_\ep(X,\mathbb{R}^{2n+1}):f[X]\cap L=\emptyset\}$ is dense for any $\ep>0$ (cf.\ \cite[Lemma 1.11.3]{EngBook}).
Since $D_\ep$ is uniformly c.e.\ open in $\ep$ by Observation \ref{obs:func-space-open}, the effective Baire category argument provides a computable map $f\in\bigcap_k D_{1/k}$, which is a computable embedding by Observation \ref{obs:basic-top-dim} (4).
\end{proof}

However, this result is not very useful in our context.
We often need to consider a $\Pi^0_1$ set $X\subseteq[0,1]^\om$ rather than a computable compactum.

\begin{theorem}\label{thm:imbedding}
Every $n$-dimensional $\Pi^0_1$ subset of $[0,1]^\om$ is computably embedded into the $n$-dimensional N\"obeling space $N^{2n+1}_n$.
\end{theorem}

We should be careful that we do not know which point is contained in a given $\Pi^0_1$ set $P$, and so $P$ is not necessarily a computable metric space, which causes a difficulty to effectivize the usual topological proof of the imbedding theorem, since the $\kappa$-mapping $\kappa:P\to\mathcal{N}(\mathcal{U})$ is not necessarily computable.
To overcome this difficulty, we shall show the approximated version of the usual argument.

An {\em $\varepsilon$-cover} is a cover $(U_k)_{k<\ell}$ such that ${\rm diam}(U_k)<{\ep}$ for any $k<\ell$.
For the effective treatment, we need an approximated version of $\ep$-mapping.
An {\em $(\varepsilon;\eta)$-mapping} is a function $g:\repsp{X}\to\repsp{Y}$ such that for any $x,y\in\repsp{X}$,
\[d(g(x),g(y))<\eta\;\Longrightarrow\;d(x,y)<\varepsilon.\]
Classically, every $\varepsilon$-mapping between compact spaces is a $(\varepsilon;\eta)$-mapping for some positive number $\eta>0$.

\begin{lemma}\label{lem:approx-imbedding}
Let $(P_s)_{s\in\om}$ be a computable approximation of an $n$-dimensional $\Pi^0_1$ subset of $[0,1]^\om$, and $L$ be an $n$-dimensional linear subspace of $\mathbb{R}^{2n+1}$.
Given a computable function $f:P_s\to\interval^{2n+1}$ and $i,j\in\om$, one can effectively find $t\geq s$, $v\in\om$, and $g:P_t\to\mathbb{I}^{2n+1}\setminus L$ such that $d(f,g)<2^{-j}$ and $g$ is a $(2^{-i};2^{-v})$-mapping.
\end{lemma}

\begin{proof}
By effective compactness of $P_j$, one can find a modulus $u(j)\geq i$ of uniform continuity of $f$, that is, if $d(p,q)<2^{-u(j)}$ then $d(f(p),f(q))<2^{-j-1}$.
By Lemma \ref{lem:dim-refinement}, one can effectively find a finite open $2^{-u(j)}$-cover $\mathcal{U}=(U_k)_{k<\ell}$ of $P$ such that ${\rm ord}(\mathcal{U})\leq n$.
Consider $\mathcal{V}$ and $\mathcal{F}$ from Lemma \ref{lem:comp-normal}.
By effective compactness, one can effectively find $t$ such that $\mathcal{U}$ covers $P_t$.

Now, by our choice of $u$, the diameter of $f[U_k]$ is at most $2^{-j-1}$.
Thus, for every $k<\ell$, there is a nonempty open ball $B_k$ whose diameter is at most $2^{-j-1}$ and $P\cap F(U_k)\subseteq f^{-1}[B_k]$, where $F(U_k)\in\mathcal{F}$ such that $F(U_k)\subseteq U_k$.
By effective compactness, this condition is c.e., and therefore, one can effectively find such $B_k$.
Choose $p_k\in B_k$ such that $p_k$'s are in general position in $\mathbb{R}^{2n+1}$ by Observation \ref{obs:general-position}, and moreover,  the linear $n$-subspace of $\interval^{2n+1}$ spanned by $(n+1)$-elements in $(p_k)_{k<\ell}$ does not intersect with $L$.
One can effectively find such $p_k$.

Let $\kappa:P_t\to N(\mathcal{V})$ be the $\kappa$-mapping determined by $\mathcal{V}$ and $(p_k)$.
Now $\kappa$ is computable.
There are only finite collection $\mathcal{L}$ of linear subspaces (spanned by at most $(n+1)$-elements in $(p_k)_{k<\ell}$).
Thus, it is easy to calculate a number $v$ such that if $L,L'\in\mathcal{L}$ are disjoint, then they have a distance $\geq 2^{-v}$ from each other.
One can check that $d(\kappa(x),\kappa(y))<2^{-v}$ implies $d(x,y)<2^{-u(j)}\leq 2^{-i}$.
\end{proof}

%
%

\begin{proof}[Proof of Theorem \ref{thm:imbedding}]
First note that $N^{2n+1}_n$ can be written as $\interval^{2n+1}\setminus\bigcup_{r}L_r$, where each $L_r$ is an at most $n$-dimensional linear subspace of $\interval^{2n+1}$.
Let $P$ be an $n$-dimensional $\Pi^0_1$ subset of $\interval^\om$, and $f_0:P_0=\interval^\om\to N^{2n+1}_n$ be a constant function.
At the $r$-th step, we assume that a $(2^{-r};2^{-v(r)})$-mapping $f_r:P_{s(r)}\to N^{2n+1}_n$ is given.
By applying Lemma \ref{lem:approx-imbedding} to $i=r+1$, $j=v(r)+1$, $s=s(r)$ and $f=f_r$, we get a $(2^{-r-1};2^{-v})$-mapping $g:P_t\to\interval^{2n+1}\setminus\bigcup_{u\leq r}L_u$ such that $d(f_r,g)<2^{-v(r)-1}$.
Define $s(r+1)=t$, $v(r+1)=\max\{v,v(r)+1\}$, and $f_{r+1}=g$.
This procedure is computable, and therefore, $f=\lim_rf_r$ gives a computable embedding of $P$ into $N^{2n+1}_n$ by Observation \ref{obs:basic-top-dim}.
\end{proof}

\subsection{The universal Menger compacta}

Let $z=(z_j)_{j\in\om}$ be a sequence of natural numbers such that $z_j\geq 3$ for all $j\in\om$.
A {\em $z$-bounded sequence} is a finite or infinite sequence $\sigma\in\om^\om$ such that $\sigma(j)<z_j$ for all $j\in\om$.
Any $z$-bounded sequence $\sigma$ determines a real $|\sigma|_z$ as follows:
\[|\sigma|_z=\sum_{j}\frac{\sigma(j)}{z^\ast_j}\mbox{, where }z^\ast_j=\prod_{k<j}z_k.\]
For instance, if $\mathbf{3}$ is a sequence consisting only of $3$, then $|\sigma|_\mathbf{3}$ is the real whose ternary expansion is $0.\sigma$.
A {\em Menger compactum} $M^m_n(z)$ is the set of all $m$-tuples $(|h_\ell|_z)_{\ell<m}\in[0,1]^m$ such that $h_\ell$ is a $z$-bounded sequence for any $\ell<m$, and, for any $j\in\om$, the following holds:
\[|\{\ell<m:h_\ell(j)\not\in\{0,z_j-1\}\}|\leq n.\]

For instance, $M^1_0(\mathbf{3})$ is the ternary Cantor set, $M^2_1(\mathbf{3})$ is the Sierpi\'nski carpet, $M^3_1(\mathbf{3})$ is the Menger sponge, $M^3_2(\mathbf{3})$ is the Sierpi\'nski sponge, and so on.
It is clear that a Menger compactum $M^m_n(z)$ is a computable compactum whenever $z$ is computable.
We effectivize the well-known fact that $M^{2n+1}_n(\mathbf{3})$ is a universal $n$-dimensional space (cf.\ \cite[Theorem 1.11.6]{EngBook}).

\begin{lemma}\label{lem:Menger-imbedding}
Let $z=(z_j)_{j\in\om}$ be a computable sequence such that $z_j\geq 3$ for all $j\in\om$, and $X$ be a $\Pi^0_1$ subset of $[0,1]^\om$ such that $X\subseteq N^m_n$.
Then, $X$ is computably embedded into $M^m_n(z)$.
\end{lemma}

\begin{proof}
At the $i$-th level, divide $[0,1]$ into $z_{i}^\ast$ many intervals $(J^i_\sigma)$ of length $1/z_i^\ast$ indexed by $z$-bounded sequences $\sigma$ of length $i$, where we ensure that $J^i_\sigma$ is the union of subintervals $(J^{i+1}_{\sigma\fr j})_{j<z_i}$, and that $j<k$ implies $\max J^{i+1}_{\sigma\fr j}\leq\min J^{i+1}_{\sigma\fr k}$.
Let $c^i_\sigma$ be the center of the interval $J^i_\sigma$.
Then define $S_i\subseteq[0,1]^m$ as the set of all $m$-tuples $(x_\ell)_{\ell<m}$ such that at least $n+1$ many $\ell$ are of the form $c^i_\sigma$ for some $\sigma$.
Since $c^i_\sigma$ is rational, $N^m_n\cap S_i$ is empty.

We will define a sequence $(h_i)$ of piecewise linear homeomorphisms on $[0,1]$, and ensure that $f_i((x_\ell)_{\ell<m})=(h_i(x_\ell))_{\ell<m}$ uniformly converges to an embedding of $X$ into $M^m_n(z)$.
Assume that we have already constructed $f_i$, and that $f_i[X]\subseteq N^m_n$.
Since $f_i[X]$ and $S_i$ are computably compact, one can effectively find a finitary approximation $L$ of $f_i[X]$ which is disjoint from $S_i$.
Then one can calculate a sufficiently small rational $\ep>0$ such that $d(S_i,L)>\ep$.
It is easy to compute a piecewise linear homeomorphism $h^i_\sigma$ of rational slope on $J^i_\sigma$ satisfying that if $x\leq c^i_\sigma-\ep$ then $h^i_\sigma(x)\in J^i_{\sigma\fr 0}$, and if $x\geq c^i_\sigma+\ep$ then $h^i_\sigma(x)\in J^i_{\sigma\fr (z_i-1)}$.
Then define $h^i$ as the union of $(h^i_\sigma)$, and then $h_i=h^i\circ h_{i-1}$.
Define $f_{i+1}$ as above, and then by rationality of $h^i_\sigma$ one can ensure that $f_{i+1}[X]\subseteq N^m_n$, which enable us to continue the induction steps.
It is easy to check that the resulting $f$ is a computable embedding of $X$ into $M^m_n(z)$ whatever $z$ is.
\end{proof}

\subsection{Fractal dimensions}

We now connect effective topological dimension theory and effective fractal dimension theory.
The latter area is extensively studied in algorithmic randomness theory, cf.\ \cite[Chapter 13]{DHBook}.
There are many known classical results connecting the relationship between topological dimension and fractal dimension.
For instance, for a dimension-theoretic notion $Dim$, topologists often found a result of the following kind:
\[\dim(E)=\inf\{Dim(Y):Y\mbox{ is homeomorphic to }E\}.\]

If $Dim=\dim_H$, then it is known as the Szpilrajn theorem, and if $Dim=\ubdim$ it is the Pontrjagin-Schnirelmann theorem, where $\dim_H$ is the Hausdorff dimension, and $\ubdim$ is the upper box-counting dimension.

Let $E$ be a compactum.
Let $\mathcal{C}_r(E)$ be the collection of all finite covers $\mathcal{U}$ of $E$ consisting of balls of diameter $\leq r$, and put $|E|_r=\min\{|\mathcal{U}|:\mathcal{U}\in\mathcal{C}_r(E)\}$.

\begin{obs}\label{obs:box-dimension-rational}
If $E\subseteq\mathbb{R}^n$ is compact, one has a cover $\mathcal{U}\in\mathcal{C}_r(E)$ consisting of rational balls that attains the minimal cardinality $|\mathcal{U}|=|E|_r$.
\end{obs}

\begin{proof}
This is because for any $U\in\mathcal{U}$ and $\mathcal{V}=\mathcal{U}\setminus\{U\}$, there is $\ep>0$ such that $d(E\setminus\bigcup\mathcal{V},\mathbb{R}^n\setminus U)>\ep$ by compactness, and therefore one can replace $U$ with a rational ball.
\end{proof}

The {\em lower and upper box-counting dimension} of $E$ are defined as follows:
\begin{align*}
\lbdim(E)=\liminf_{r\to 0}\frac{\log |E|_r}{\log(r^{-1})},& &\ubdim(E)=\limsup_{r\to 0}\frac{\log |E|_r}{\log(r^{-1})}.
\end{align*}

It is easy to see that $\dim_H(X)\leq\underline{\dim}_B(X)$.
For a polyhedron, we have $\dim(X)=\dim_H(X)=\ubdim(X)=\lbdim(X)$.

It is not hard to effectivize the Pontrjagin-Schnirelmann theorem by a straightforward argument. 
%
%
%
Then it is natural to ask whether one can replace the lower box dimension $\lbdim(X)$ in the effective Pontrjagin-Schnirelmann Theorem with the upper box dimension $\ubdim$ or the packing dimension $\dim_P$.
In the classical setting, Joyce \cite{Joy98} has shown the Pontrjagin-Schnirelmann Theorem for the packing dimension $\dim_P$ by a slight modification of the standard $\kappa$-mapping argument.
Luukkainen \cite{Luu98} has shown the Pontrjagin-Schnirelmann for the Assouad dimension $\dim_A$.

The Assouad dimension is a modification of the upper box-counting dimension.
Note that if $\ubdim(E)\leq s$ then for any $\ep>0$ and for any sufficiently small $r$, $\log|E|_r\leq (s+\ep)\log(r^{-1})$, that is, $|E|_r\leq r^{-s+\ep}$ holds.
The {\em Assouad dimension} of $E$, denoted by $\dim_A(E)$, is the infimum of $s$ such that there are $c$ and $\rho$ such that for any positive reals $r<R<\rho$, the following holds.
\[\sup_{x\in E}|E\cap B_R(x)|_r\leq c\left(\frac{R}{r}\right)^s.\]

It is easy to see the following inequalities.
\[\dim(E)\leq\dim_H(E)\leq\dim_P(E)\leq\ubdim(E)\leq\dim_A(E).\]

\begin{fact}[Luukkainen {\cite[Lemma 3.8]{Luu98}}]\label{fact:Luukainen}
Let $z=(z_n)_{n\in\om}\in\om^\om$ be such that $\lim_nz_n=\infty$.
Then $\dim_A(M^m_n(z))=n$.
\end{fact}

\begin{cor}\label{cor:Assouad-dimension}
Every $n$-dimensional $\Pi^0_1$ subset of $[0,1]^\om$ is computably embedded into a computable compact subset of $\mathbb{R}^{2n+1}$ of the Assouad dimension $n$.
\end{cor}

\begin{proof}
By Theorem \ref{thm:imbedding}, Lemma \ref{lem:Menger-imbedding}, and Fact \ref{fact:Luukainen}.
\end{proof}

\section{The metric $tt$-degree theory}

\subsection{Topological dimension of points}

As mentioned before, McNicholl-Rute (Fact \ref{fact:MR-computable-arc}) has shown that a point $x\in\mathbb{R}^2$ is contained in a computable arc if and only if $x$ is $tt$-equivalent to a point in $\mathbb{R}$.
Then, it is natural to ask a generalized question:
Which point in the Hilbert cube can be $tt$-equivalent to a point in $\mathbb{R}^n$ for some $n\in\om$?

\subsubsection{Universal Menger compacta}

We first give a characterization of the $tt$-degrees of $n$-dimensional points of computable compacta, that is, the $n$-dimensional points are exactly those of $M^{2n+1}_n(\mathbf{3})$-$tt$-degrees.

\begin{theorem}\label{thm:universal-menger}
The following are equivalent for a point $x\in[0,1]^\om$ and $n\in\om$.
\begin{enumerate}
\item $x$ is $tt$-equivalent to a point in an $n$-dimensional computable compactum.
\item $x$ is $tt$-equivalent to a point in the universal Menger compactum $M^{2n+1}_n(\mathbf{3})$.
\item $x$ is contained in an $n$-dimensional $\Pi^0_1$ subset of $[0,1]^\om$.
\end{enumerate}
\end{theorem}

\begin{proof}
(2)$\Rightarrow$(1): This is because $M^{2n+1}_n(\mathbf{3})$ is an $n$-dimensional computable compactum.
(1)$\Rightarrow$(3): 
Let $y$ be a point in an $n$-dimensional computable compactum $\repsp{Y}$.
By Lemma \ref{lem:main-lemma}, there is a computable embedding $\Phi$ of a $\Pi^0_1$ set $P\subseteq[0,1]^\om$ into $\repsp{Y}$ such that $x\in P$ and $\Phi(x)=y$.
Since $P$ is homeomorphic to the subset $\Phi[P]$ of the $n$-dimensional space $\repsp{Y}$, $P$ is also $n$-dimensional.
(3)$\Rightarrow$(2):
Let $P$ be an $n$-dimensional $\Pi^0_1$ subset of $[0,1]^\om$ containing $x$.
By Theorem \ref{thm:imbedding} and Lemma \ref{lem:Menger-imbedding}, there is a computable embedding $\Phi$ of $P$ into $M^{2n+1}_n(\mathbf{3})$.
By Observation \ref{obs:basic-top-dim}, the embedded image $\Phi[P]$ is $\Pi^0_1$ in $M^{2n+1}_n(\mathbf{3})$.
Hence, by Lemma \ref{lem:main-lemma}, $x$ is $tt$-equivalent to a point in $M^{2n+1}_n(\mathbf{3})$.
\end{proof}

The above result can also be seen as an effectivization of Jayne-Rogers' result \cite{JR79b} saying that there is only one universal $n$-dimensional compactum up to first level Borel isomorphism.

\begin{cor}\label{cor:n-dim-uniform-degree}
The following are equivalent for $x\in[0,1]^\om$ and $n\in\om$:
\begin{enumerate}
\item $x$ is $tt$-equivalent to a point in a Euclidean space.
\item $x$ is $tt$-equivalent to a point in a finite dimensional computable compactum.
\item $x$ is contained in a finite dimensional $\Pi^0_1$ subset of $[0,1]^\om$.
\end{enumerate}
\end{cor}

\begin{proof}
This follows from Theorem \ref{thm:universal-menger} and the fact that $M^{2n+1}_n(\mathbf{3})\subseteq\mathbb{R}^{2n+1}$.
\end{proof}

We now consider two hierarchies of dimension of points:
A point $x\in\repsp{X}$ is {\em $n$-Euclidean} if it is $tt$-equivalent to a point in $\mathbb{R}^n$.
A point $x\in\repsp{X}$ is {\em $n$-dimensional} if it is $tt$-equivalent to a point in a $n$-dimensional computable compactum.
We also say that a point $x$ is {\em finite dimensional} if it is $n$-dimensional for some $n\in\om$.
By Corollary \ref{cor:n-dim-uniform-degree}, $x$ is finite dimensional iff $x$ is $n$-Euclidean for some $n\in\om$.

Note that McNicholl-Rute's result says that a point $x\in\mathbb{R}^2$ is contained in a computable arc iff $x$ is $1$-Euclidean.
The following is a trivial consequence of Theorem \ref{thm:universal-menger}.

\begin{obs}
Every $n$-dimensional point is $(2n+1)$-Euclidean.
\end{obs}

It is not hard to show that an $n$-dimensional point is not necessarily $(2n)$-Euclidean, by using the classical topological fact that there is an $n$-dimensional space which cannot be embedded into $\mathbb{R}^{2n}$ as follows.

\begin{prop}\label{prop:Menger-non-2n}
There is an $n$-dimensional point in $\interval^\om$ which is not $(2n)$-Euclidean.
\end{prop}

\begin{proof}
Here, we give an explicit description of such a point.
Let $T_n$ be the set of all ternary sequences in $3^{2n+1}$ of length $(2n+1)$ containing at most $n$ many $1$'s, and $m(n)$ be the cardinality of $T_n$, that is,
\[m(n)=\sum_{k=0}^n\binom{2n+1}{k}2^{2n-k+1}.\]
For instance, $m(1)=20$, $m(2)=192$, and so on.
Fix a bijection $c\mapsto (b^c_0,\dots,b^c_{2n})$ between $m(n)$ and $T_n$.
Then, given $z\in(m(n))^\om$, we can get a $(2n+1)$-tuple $h(z)=(x^z_0,\dots,x^z_{2n})$ of reals as follows:
\[x^z_k=0.b^{z(0)}_kb^{z(1)}_kb^{z(2)}_k\dots\]

We claim that
\begin{itemize}
\item[] {\em if $z$ is a weakly $1$-generic sequence in $(m(n))^\om$, then the  $(2n+1)$-tuple $(x^z_0,\dots,x^z_{2n})\in\interval^{2n+1}$ is $n$-dimensional, but not $(2n)$-Euclidean.}
\end{itemize}
Since $z$ is not periodic, we have $h(z)\in M^{2n+1}_n$.
In particular, $h(z)$ is $n$-dimensional.
Note that if $U$ is a c.e.\ open subset of $\interval^{2n+1}$, dense in $M^{2n+1}_n$, then $h^{-1}[U]$ is dense c.e.\ open in $(m(n))^\om$.
Thus, $z\in h^{-1}[U]$, and therefore, $h(z)\in M^{2n+1}_n\cap U$.
That is, $h(z)$ is weakly $1$-generic in $M^{2n+1}_n$.

If $h(z)$ is $tt$-equivalent to a point in $\mathbb{R}^{2n}$, then by Lemma \ref{lem:main-lemma}, there is a $\Pi^0_1$ set $P\subseteq M^{2n+1}_n$ with $h(z)\in P$ such that $P$ embeds into $\mathbb{R}^{2n}$.
By weak $1$-genericity of $h(z)$, $P$ must contain a nonempty interior.
However, any nonempty open subset of $M^{2n+1}_n$ contains a copy of $M^{2n+1}_n$.
Thus, this gives an embedding of $M^{2n+1}_n$ into $\mathbb{R}^{2n}$, which contradicts the fact that there is an $n$-dimensional space which cannot be embedded into $\mathbb{R}^{2n}$.
This concludes that $h(z)$ is not $(2n)$-Euclidean.
\end{proof}

%

In particular, such a point is proper $n$-dimensional.
Here, we say that a point is {\em proper $(n+1)$-dimensional} if it is $(n+1)$-dimensional, but not $n$-dimensional.
By $\dim(x)$ we denote the $n\in\om$ such that $x$ is $n$-dimensional, but not $m$-dimensional for all $m<n$.
In particular, $\dim(x)=n$ iff $x$ is proper $n$-dimensional.
One can see that every $n$-dimensional computable compactum contains a proper $n$-dimensional point.
Indeed, we have the following.

\begin{obs}\label{obs:point-to-set-topdimension}
Let $P$ be a $\Sigma^0_2$ subset of $\interval^\om$.
Then, $\dim(P)=\sup_{x\in P}\dim(x)$.
\end{obs}

\begin{proof}
Let $P$ be a countable union of $\Pi^0_1$ sets $(P_i)_{i\in\om}$.
Then $x\in P$ implies that $x\in P_i$ for some $i\in\om$.
By Theorem \ref{thm:universal-menger} (3)$\Rightarrow$(1), $x\in P_i$ implies $\dim(x)\leq\dim(P_i)\leq\dim(P)$.
Conversely, assume that $\dim(x)\leq n$ for all $x\in P$.
Then there is an at most $n$-dimensional $\Pi^0_1$ set $Q_x$ containing $x$.
However, there are countably many $\Pi^0_1$ sets, and thus $P$ is a union of countably many at most $n$-dimensional closed subsets.
Thus, by the sum theorem (cf.\ \cite[Theorem III.2]{HWBook}), we have $\dim(P)\leq n$.
\end{proof}

\subsubsection{Genericity}

We give a characterization of proper $n$-dimensionality for $n$-Euclidean points.
A point $x\in X$ is {\em weakly $1$-generic} if $x$ contains no nowhere dense $\Pi^0_1$ subset of $X$.

\begin{obs}\label{obs:proper-n-dim-w1g}
A point $x\in\mathbb{R}^{n}$ is proper $n$-dimensional if and only if $x$ is weakly $1$-generic in $\mathbb{R}^{n}$.
\end{obs}

\begin{proof}
Note that a subset of $\mathbb{R}^{n+1}$ is $(n+1)$-dimensional if and only if it has an nonempty interior (see \cite[Theorem IV.3]{HWBook}).
Thus, a closed subset of $\mathbb{R}^{n+1}$ is $(n+1)$-dimensional if and only if it is somewhere dense.
Therefore, $x\in\mathbb{R}^{n+1}$ is weakly $1$-generic if and only if $x$ is not contained in an $n$-dimensional $\Pi^0_1$ subset of $\mathbb{R}^{n+1}$, that is, $x$ is not $n$-dimensional by Theorem \ref{thm:universal-menger}.
\end{proof}

A topological space is countable dimensional if it is a countable union of finite dimensional subspaces.
We say that a point $x\in\interval^\om$ is {\em countable dimensional} if it is contained in a countable dimensional $\Pi^0_1$ set.
%
%
A point $x\in\interval^\om$ is {\em total} if there is $y\in 2^\om$ such that $y\equiv_Tx$.
Kihara-Pauly \cite{KP} has shown that a Polish space $\repsp{X}$ is countable dimensional iff some relativization makes all points in $\repsp{X}$ be total; however,

\begin{obs}\label{obs:non-ctbl-dim}
Every weakly $1$-generic point in $\interval^\om$ is total, but not countable dimensional.
\end{obs}

\begin{proof}
Note that every countable dimensional closed set is nowhere dense in $\interval^\om$ since every nonempty open subset of $\interval^\om$ contains a copy of $\interval^\om$, which is not countable dimensional.
Moreover, all coordinates of a weakly $1$-generic point $x$ in $\interval^\om$ must be irrational, which clearly implies that $x$ is total.
This concludes the proof.
\end{proof}

This observation reflects the fact that the total degrees form a countable dimensional $\Sigma^0_3$ set, but are not covered by a countable dimensional $F_\sigma$ set.

Pol-Zakrzewski \cite[Remark 5.5]{PZ12} studied the forcing $\mathbb{P}_I$ obtained from the $\sigma$-ideal $I$ generated by finite dimensional closed subsets of a fixed compactum.
Zapletal \cite{Zap14} used this forcing to solve Fremlin's old problem asking whether there exists a ``{\em half-Cohen forcing}.''
That is, any $\mathbb{P}_I$-generic extension $V[G]$ of $V\models {\rm ZFC}$ does not contain a Cohen real over $V$ (indeed, $V[G]$ is a minimal extension of $V$), but any $P_I$-generic extension $V[G][H]$ of $V[G]$ must contain a Cohen real over $V$.
This forcing is further studied by \cite{Pol15,PZ16}.
Recall that the notion of weak $1$-genericity is an effective version of Cohen genericity.
Then, non-finite-dimensionality of a point can be thought of as an effective version of genericity w.r.t.\ this half-Cohen forcing $\mathbb{P}_I$, so one might call a non-finite-dimensional point {\em half-generic}.

We say that $f:\om\to\om$ is infinitely often equal to $g:\om\to\om$ if there are infinitely many $n\in\om$ such that $f(n)=g(n)$.
A function is {\em computably i.o.e.}\ if it is infinitely often equal to all computable functions.
Zapletal's proof \cite[Lemma 2.2]{Zap14} shows that if $x\in[0,1]^\om$ is not contained in any arithmetically-coded closed subset of $[0,1]^\om$ then there is an $x$-arithmetically definable function $f$ which is arithmetically i.o.e., that is, infinitely often equal to all arithmetically definable functions.
However, there is no computable analogue of Zapletal's result.

\begin{prop}
There is a non-countable-dimensional point which computes no computably i.o.e.\ function.
\end{prop}

\begin{proof}
We say that a $x$ has a PA-degree if it computes a complete consistent extension of Peano Arithmetic.
Miller \cite{Mil04} showed that any PA-degree bounds a nontotal degree.
By the hyperimmune-free basis theorem, cf.\ \cite[Theorem 2.9.11]{DHBook}, there is a hyperimmune-free PA-degree.
Therefore, there is a nontotal point $x\in[0,1]^\om$ such that every $f\leq_Tx$ is bounded by a computable function, but such an $f$ cannot be computably i.o.e.
\end{proof}

The following is an analogue of the argument in Zapletal \cite[Section 3]{Zap14}.

\begin{prop}\label{prop:nonz-fin-w1g}
Let $x$ be a non-zero-dimensional, finite-dimensional point.
Then, there is a weakly $1$-generic real $y\leq_{tt}x$.
\end{prop}

\begin{proof}
Since $x$ is finite-dimensional, one can assume $x=(x_i)_{i<n}\in[0,1]^n$.
We claim that $x_i$ is weakly $1$-generic for some $i<n$.
Otherwise, for any $i<n$ there is a nowhere dense $\Pi^0_1$ set $P_i\subseteq[0,1]$ such that $x_i\in P_i$.
However, $\prod_{i<n}P_i$ is a zero-dimensional $\Pi^0_1$ subset of $[0,1]^n$, and therefore $x$ must be zero-dimensional.
\end{proof}

It is not hard to check that a real in $[0,1]$ is weakly $1$-generic if and only if its binary expansion is weakly $1$-generic in $2^\om$.
Moreover, a $T$-degree $\mathbf{d}$ contains a weakly $1$-generic real in $2^\om$ if and only if $\mathbf{d}$ is hyperimmune (see \cite[Corollary 2.24.19]{DHBook}), where a $T$-degree $\mathbf{d}$ is hyperimmune if there is a $\mathbf{d}$-computable function which dominates all computable functions.
In particular, Proposition \ref{prop:nonz-fin-w1g} implies that every non-zero-dimensional, finite-dimensional point has a hyperimmune $T$-degree.

\subsubsection{Inside $T$-degrees}

Classically, it is known that a $T$-degree $\mathbf{d}$ is hyperimmune-free iff $\mathbf{d}$ contains only one $tt$-degree (see \cite[Theorem VI.6.18]{OdiBook}).
It is easy to extend this classical fact as follows.
Recall that a $T$-degree $\mathbf{d}$ is total if there is $x\in 2^\om$ of $T$-degree $\mathbf{d}$.
The following strengthens the observation mentioned in the previous paragraph.

\begin{prop}\label{prop:hypim-free-str}
Let $\mathbf{d}$ be a hyperimmune-free total $T$-degree.
Then, for any $x\in\mathbf{d}$ and $y\in[0,1]^\om$, $y\leq_Tx$ if and only if $y\leq_{tt} x$.
\end{prop}

\begin{proof}
%
%
%
%
%
Let $\mathbf{d}$ be a hyperimmune-free total $T$-degree.
Let $x\in\mathbf{d}$ in a computable metric space $\repsp{X}$.
Since $\mathbf{d}$ is total, there is a Cauchy name $p\in\om^\om$ of $x$ such that $p\equiv_Tx$.
Assume that $y\leq_Tx$ via a partial computable function $f$.
Let $\Phi$ be a realizer of $f$.
Since $p$ is computably bounded, there is a computable increasing function $s:\om\to\om$ such that $\Phi^{p\upto s(n)}_{s(n)}(n)\downarrow$ and $p(n)<s(n)$ for all $n\in\om$.
Let $Q$ be the set as in the proof of Proposition \ref{prop:tt-time}.
Then, $p\in Q$.
Moreover the restriction $f\upto\delta[Q]$ clearly satisfies the premise of Proposition \ref{prop:tt-time} (2).
Therefore, by Proposition \ref{prop:tt-time}, $f\upto\delta[Q]$ can be extended to a computable function $g:\subseteq\repsp{X}\to\repsp{Y}$ whose domain is $\Pi^0_1$.
Since $p\in Q$ and thus $x\in\delta[Q]$, we conclude $y\leq_{tt}x$.
\end{proof}

\begin{cor}
A hyperimmune-free total $T$-degree consists only of zero-dimensional points.
\end{cor}

\begin{proof}
By definition, every total $T$-degree $\mathbf{d}$ contains a zero-dimensional point $p$.
If $x\in\mathbf{d}$ then $x\equiv_Td$, and if $\mathbf{d}$ is hyperimmune-free, then by Proposition \ref{prop:hypim-free-str}, we have $x\equiv_{tt}p$.
Hence, $x$ is zero-dimensional.
\end{proof}

Miller's observation \cite[Proposition 5.3]{Mil04} on continuous degrees implies the following analogue of Proposition \ref{prop:hypim-free-str} for nontotal $T$-degrees.

\begin{prop}\label{prop:nontotal-T-vs-tt}
Let $\mathbf{d}$ be a nontotal $T$-degree.
For any $x\in\mathbf{d}$, and a finite dimensional point $y$, $y\leq_Tx$ if and only if $y\leq_{tt} x$.
\end{prop}

\begin{proof}
Assume that $y\leq_Tx$.
By finite dimensionality, there is $z\equiv_{tt}y$ such that $z=(z_i)_{i<n}\in\interval^n$ for some $n\in\om$.
Define $\hat{x}(i,j,k)=x(i)b_j+b_k$, where $b_e$ is the $e$-th rational.
Clearly, $\hat{x}\equiv_{tt}x$.
Miller \cite[Proposition 5.3]{Mil04} showed (the contrapositive of) the following:
Whenever $x$ is non-total, for any $p\in\interval$, if $p\leq_T\hat{x}$, then there is $m$ such that $\hat{x}(m)=p$.
Since $z_i\leq_T\hat{x}$ for any $i<n$, there are $m_0,\dots,m_{n-1}$ such that $z=(\hat{x}(m_0),\dots,\hat{x}(m_{n-1}))$, which is clearly $tt$-reducible to $\hat{x}$.
Thus, we have $y\equiv_{tt}z\leq_{tt}\hat{x}\equiv_{tt}x$.
\end{proof}

%


\subsubsection{Quasi-minimality}

Miller \cite{Mil04} showed that there is no quasi-minimal continuous degree, that is, there is no noncomputable point $x\in\interval^\om$ such that every $x$-computable point $y\in 2^\om$ is computable.
Then, does there exist a metric $tt$-degree which is quasi-minimal in $tt$-degrees?
For computable metric spaces $\repsp{X}$ and $\repsp{Y}$, we say that $x\in\repsp{X}$ is {\em $\repsp{Y}$-quasi-minimal} if $x$ is noncomputable, and every $y\in\repsp{Y}$ with $y\leq_{tt}x$ is computable.
It is easy to construct a $2^\om$-quasi-minimal $\mathbb{R}$-uniform degree.
A point $x\in\repsp{X}$ is {\em $1$-generic} if it is not contained in the boundary $\partial U$ of a c.e.\ open set $U\subseteq\repsp{X}$.

\begin{obs}
Every $1$-generic point in $\mathbb{R}$ is $2^\om$-quasi-minimal.
\end{obs}

\begin{proof}
Assume that $y\leq_{tt}x$ for $y\in 2^\om$ and $x\in\mathbb{R}$.
Then there is a $\Pi^0_1$ set $P\subseteq\mathbb{R}$ and $\Phi:P\to 2^\om$ such that $x\in P$ and $\Phi(x)=y$.
If $x$ is $1$-generic, then $P$ contains a nondegenerated interval $J\subseteq P$ such that $x\in J$ (otherwise, $x$ is contained in the closure of $\mathbb{R}\setminus P$, which contradicts $1$-genericity of $x$).
Then $\Phi$ is constant on $J$ since $J$ is connected while $2^\om$ is totally disconnected.
This means that $\Phi(z)=y$ for any $z\in J$ since $x\in J$.
By choosing a rational $q\in J\cap\mathbb{Q}$, we get that $y=\Phi(q)\leq_{tt}q$ is computable. 
\end{proof}

This is not true for higher dimensional case.

\begin{obs}
No finite dimensional point is $\mathbb{R}$-quasi-minimal.
\end{obs}

\begin{proof}
This is because a point in an Euclidean space is computable if and only if all of its coordinates are computable.
\end{proof}

By Proposition \ref{prop:nontotal-T-vs-tt}, if $x\in\interval^\om$ is nontotal, then $x$ has no $\mathbb{R}^n$-quasi-minimal $tt$-degree.
Therefore, an $\mathbb{R}$-quasi-minimal uniform degree has to be total, but not finite dimensional if it exists.
By Observation \ref{obs:non-ctbl-dim}, every weakly $1$-generic point in $\mathbb{I}^\om$ is total, and not countable dimensional.
However, no weakly $1$-generic point can be $\mathbb{R}$-quasi-minimal.

\begin{obs}
No weakly $1$-generic point in $\mathbb{I}^\om$ is $\mathbb{R}$-quasi-minimal.
\end{obs}

\begin{proof}
All of coordinates of a weakly $1$-generic point are noncomputable.
\end{proof}

\begin{question}
Does there exist an $\mathbb{R}$-quasi-minimal uniform degree?
\end{question}

\subsection{Effective fractal dimension}

In computability theory (particularly in algorithmic randomness theory), it is usual to consider the algorithmic dimension (the algorithmic information density) of a point in the context of fractal dimension; see \cite[Section 13]{DHBook}.
The notions of Kolmogorov complexity and algorithmic dimension in a Euclidean space has been studied in \cite{LuMa08,LuLu17,LuSt17a,LuSt17b}, for instance.
In this section, we compare our notion of topological dimension of points and the notions of effective fractal dimension of points.

\subsubsection{Zero dimensional spaces}

Let $C$ and $K$ denote the plain and the prefix-free Kolmogorov complexity, respectively.
The {\em effective Hausdorff dimension} of an infinite binary sequence $x\in 2^\om$ is defined as follows.
\[\dim_H(x)=\liminf_{n\to\infty}\frac{K(x\upto n)}{n}=\liminf_{n\to\infty}\frac{C(x\upto n)}{n}.\]

Similarly, the {\em effective packing dimension} of $x\in 2^\om$ is defined as follows.
\[\dim_P(x)=\limsup_{n\to\infty}\frac{K(x\upto n)}{n}=\limsup_{n\to\infty}\frac{C(x\upto n)}{n}.\]

There are more effective versions of Hausdorff and packing dimension.
A machine is a partial computable function whose domain is a subset of $2^{<\om}$.
Let $C_M$ and $K_M$ denote the plain and the prefix-free Kolmogorov complexity relative to a machine $M$, respectively.
A {\em decidable machine} \cite[Definition 7.3.1]{DHBook} is a machine $M$ having the decidable halting problem, that is, ${\rm dom}(M)$ is computable.
A {\em computable measure machine} \cite[Definition 7.1.14]{DHBook} is a prefix-free machine $M$ whose halting probability $\Omega_M$ is computable.
Note that every computable measure machine is decidable.

\begin{obs}\label{obs:dec-vs-cmm}
For every decidable machine $M$, there is a computable measure machine $N$ such that $K_N(\tau)\leq C_M(\tau)+2\log C_M(\tau)+O(1)$.
\end{obs}

\begin{proof}
One can assume that $M$ is a total machine by extending the domain.
Let $b_n$ denotes the binary presentation of $n\in\om$, and let $b^+_n$ be the result by inserting $0$ into each of consecutive bits in $b_n$, that is, 
\[b^+_n=b_n(0)0b_n(1)0b_n(2)0\dots 0b_{n}(|b_n|-1)0.\]
Then, define $N(b^+_{|\sigma|}11\sigma)=M(\sigma)$.
It is clear that $N$ is a computable measure machine, and satisfies the desired inequality.
\end{proof}

Unlike an (undecidable) machine, which only ensures the existence of a decompression algorithm, one of the most important features of a decidable machine is the relationship with a compression algorithm in the real world, cf.\ Bienvenu-Merkle \cite{BM07}.
Formally, a {\em compression algorithm} or a {\em compressor} \cite{BM07} is a partial computable injection $M:\subseteq 2^{<\om}\to 2^{<\om}$ such that the domain and the image of $M$ are computable.
Then the Kolmogorov complexity $C_M(\sigma)$ w.r.t.\ a compression algorithm $M$ is defined as the length of $M(\sigma)$.
There is only a constant difference between the Kolmogorov complexities w.r.t.\ a decidable machine and a compression algorithm; see \cite{BM07}. 

Let $\mathcal{M}_{\rm cm}$ and $\mathcal{M}_{\rm dm}$ be the collections of all computable measure machines and all decidable machines (or all compression algorithms), respectively.
Then the {\em Schnorr Hausdorff dimension} of $x\in 2^\om$, denoted by $\dim_H^{\rm Sch}(x)$, is given as follows (cf.\ \cite[Theorem 13.15.8]{DHBook}):
\[\dim_H^{\rm Sch}(x)=\inf_{M\in\mathcal{M}_{\rm cm}}\liminf_{n\to\infty}\frac{K_M(x\upto n)}{n}=\inf_{M\in\mathcal{M}_{\rm dm}}\liminf_{n\to\infty}\frac{C_M(x\upto n)}{n}.\]

The latter equality follows from Observation \ref{obs:dec-vs-cmm}.
Similarly, the {\em Schnorr packing dimension} of $x\in 2^\om$, denoted by $\dim_P^{\rm Sch}(x)$, is given as follows (cf.\ \cite[Theorem 13.15.9]{DHBook}):
\[\dim_P^{\rm Sch}(x)=\inf_{M\in\mathcal{M}_{\rm cm}}\limsup_{n\to\infty}\frac{K_M(x\upto n)}{n}=\inf_{M\in\mathcal{M}_{\rm dm}}\limsup_{n\to\infty}\frac{C_M(x\upto n)}{n}.\]

It is clear that $\dim_H(x)\leq\dim_H^{\rm Sch}(x)$ and $\dim_P(x)\leq\dim_P^{\rm Sch}(x)$.
We consider yet another notion of effective fractal dimension.

\begin{prop}\label{prop:Kurtz-Hausdorff}
Let $0\leq s\leq 1$ be a computable real.
The following are equivalent for any $x\in 2^\om$.
\begin{enumerate}
\item $x$ is contained in a Hausdorff $s$-null $\Pi^0_1$ set.
\item There are a computable measure machine $M$ and a computable order $g$ such that
\[(\forall k\in\om)(\exists n\in[g(k),g(k+1)))\quad\frac{K_M(x\upto n)+k}{n}<s.\]
\end{enumerate}
\end{prop}

\begin{proof}
Kihara-Miyabe \cite{KiMi14} introduced the following notion:
A set $E\subseteq 2^\om$ is {\em Kurtz $s$-null} if there is a computable sequence $(C_n)_{n\in\om}$ of finite sets of strings
such that $E\subseteq[C_n]$ and $\sum_{\sigma\in C_n}2^{-s|\sigma|}\leq 2^{-n}$ for all $n\in\om$.
If $P\subseteq 2^\om$ is $\Pi^0_1$, by effective compactness, $P$ is Hausdorff $s$-null, iff $P$ is Kurtz $s$-null (c.f.\ the proof of \cite[Theorem 13.6.1]{DHBook} for the details).
Thus, it suffices to show that $x$ is contained in a Kurtz $s$-null $\Pi^0_1$, iff the condition (2) holds for $x$.
Kihara-Miyabe \cite[Theorem 5.2]{KiMi14} showed that $\{x\}$ is Kurtz $s$-null iff (2) holds for $x$.
If $\{x\}$ is Kurtz $s$-null via $(C_n)_{n\in\om}$, then $\bigcap_{n\in\om}[C_n]$ is a Kurtz $s$-null $\Pi^0_1$ set containing $x$.
This concludes the proof.
\end{proof}

By Observation \ref{obs:dec-vs-cmm}, one can replace a computable measure machine in (2) with a decidable machine (or a compression algorithm).
If $x$ satisfies the condition (2), then we call $x$ {\em computably often $s$-compressible} (c.o.\ $s$-compressible).
Define $\kdim(x)$ as the infimum of the set of $s\in[0,\infty)$ such that $x$ is c.o.\ $s$-compressible.
Obviously, $\dim_H^{\rm Sch}(x)\leq\kdim(x)\leq\dim_P^{\rm Sch}(x)$.

\begin{cor}
Let $0\leq s\leq 1$ be a computable real.
Then, $x\in 2^\om$ is contained in a $\Pi^0_1$ set of Hausdorff dimension $\leq s$ iff $\kdim(x)\leq s$.
\end{cor}

\begin{proof}
Note that $\dim_H(x)\leq s$ iff for any $t>s$, there is an effective $\mathcal{H}^t$-null set containing $x$.
By Proposition \ref{prop:Kurtz-Hausdorff}, it is equivalent to saying that $x$ is c.o.\ $t$-compressible for all $t>s$, that is, $\kdim(x)\leq s$.
\end{proof}

\begin{obs}
Let $0\leq s\leq 1$ be a computable real.
Then, there is $x\in 2^\om$ such that $\dim_H(x)=\kdim(x)=s$.
\end{obs}

\begin{proof}
It is clear that $\dim_H(x)\leq\kdim(x)$.
It is easy to construct a $\Pi^0_1$ set $Q\subseteq 2^\om$ such that $\dim_H(Q)=s$.
Then, for any $t<s$, $Q$ is not covered by the open set $U_t$ generated by $\{\sigma:K(\sigma)\leq t|\sigma|-k\}$ for large $k$.
By compactness of $Q$, there is $x\in Q\setminus\bigcup_jU_{s-2^{-j}}$.
Then, for any $j$, there is $k$ such that $K(x\upto n)>(s-2^{-j})n-k$ for all $n$.
This implies that $\dim_H(x)\geq s-2^{-j}$ for any $j$, and thus $\dim_H(x)\geq s$.
Moreover, $x\in Q$ implies $\kdim(x)\leq s$.
\end{proof}

\subsubsection{Finite dimensional spaces}

Given $x\in\mathbb{R}^n$, the {\em Kolmogorov complexity of $x$ at precision $r$} (cf.\ \cite{LuMa08,LuLu17,LuSt17a,LuSt17b}) is defined as follows:
\[K_r(x)=\min\{K(q):q\in\mathbb{Q}^n\mbox{ and }d(x,q)<2^{-r}\}.\]


We also define the Kolmogorov complexity w.r.t.\ a machine $M$ at precision $r$ in a straightforward manner, denoted by $C_{M,r}$ and $K_{M,r}$.
Then, the effective Hausdorff and packing dimension of $x\in\mathbb{R}^n$ is defined as $\dim_H(x)=\liminf_rK_r(x)/r$ and $\dim_P(x)=\limsup_rK_r(x)/r$.
The Schnorr Hausdorff and packing dimensions $\dim_H^{\rm Sch}$ and $\dim_P^{\rm Sch}$ are also defined in a similar manner.

\begin{prop}\label{prop:weakly1gen-Haus}
If $x\in\mathbb{R}^{n}$ is weakly $1$-generic, then $\dim_H(x)=0$.
\end{prop}

\begin{proof}
For any $k\in\om$, the set $S_k$ of all $x$ such that $K_{kr}(x)\leq r$ for some $r\geq k$ is dense, since if $p$ is computable then $K_{kr}(p)\leq r$ for almost all $r$.
We claim that $S_k$ is c.e.\ open.
To see this, consider the c.e.\ set $C_t=\{q\in\mathbb{Q}^n:K(q)\leq t\}$, which generates the c.e.\ open set $C_{t,r}=\bigcup\{B(q;2^{-r}):q\in C_t\}$.
By definition, $K_r(x)\leq t$ iff $x\in C_{t,r}$.
Hence, $S_k=\bigcup_{r\geq k}C_{r,kr}$ is c.e.\ open.
Thus, if $x$ is weakly $1$-generic, then $x\in S_k$ for any $k$, which implies that for any $k$ there is $r$ such that $K_{kr}(x)/kr\leq 1/k$.
Hence, we have $\liminf_rK_r(x)/r=0$, that is, $\dim_H(x)=0$.
\end{proof}

A point $x\in\repsp{X}$ is {\em weakly $n$-generic} if it is contained in any dense $\Sigma^0_n$-open set.
A point $x\in\repsp{X}$ is {\em $n$-generic} if it is not contained in the boundary $\partial U$ of a dense $\Sigma^0_n$-open set $U\subseteq\repsp{X}$.

\begin{prop}
If $x\in\mathbb{R}^{n}$ is weakly $2$-generic, then $\dim_P(x)=n$.
\end{prop}

\begin{proof}
For any $k\in\om$, the set $S_k$ of all $x$ such that $K_{r+1}(x)/r\geq n(1-1/k)$ for some $r\geq k$ is dense, since if $p$ is random then $K_{r}(p)\geq nr-O(1)$.
Again consider the c.e.\ set $C_t=\{q\in\mathbb{Q}^n:K(q)\leq t\}$.
Consider the set $E_{t,r}$ of all $x$ such that $d(x,q)>2^{-r}$ for all $q\in C_t$.
Note that $(C_t)_{t\in\om}$ is a $\emptyset'$-computable sequence of finite sets.
Then $E_{t,r}$ is open since $C_t$ is finite, and the sequence $(E_{t,r})$ is $\emptyset'$-computable.
Clearly, $S_k\subseteq S^\ast_k:=\bigcup_{r\geq k}E_{r n(1-1/k)}$, and thus the latter set is a dense $\emptyset'$-c.e.\ open set.
Hence, if $x$ is weakly $2$-generic then $x\in S^\ast_k$ for all $k$, which implies that $\limsup_nK_r(x)/r=n$, that is, $\dim_P(x)=n$.
\end{proof}

\begin{prop}\label{prop:2-gen-Sch-dim}
If $x\in\mathbb{R}^n$ is $2$-generic, then $\dim_{H}^{\rm Sch}(x)=0$, but $\dim(x)=n$.
\end{prop}

\begin{proof}
Let $x=(x_i)_{i<n}$ be $2$-generic.
By Observation \ref{obs:proper-n-dim-w1g}, $\dim(x)=n$.
By the standard property of Cohen genericity, $(x_i)_{i<n}$ is mutually $2$-generic.
Since $x_i$ is irrational, $x_i$ has the unique binary expansion $\tilde{x}_i$.
Then $(\tilde{x}_i)_{i<n}$ is mutually $2$-generic, and therefore, $\tilde{x}=\bigoplus_{i<n}\tilde{x}_i$ is $2$-generic in $2^\om$.
Clearly $\tilde{x}\equiv_Tx$.
By $2$-genericity of $\tilde{x}$, we have $\tilde{x}''\leq_T\tilde{x}\oplus\emptyset''$ \cite[Theorem 2.24.3]{DHBook}; hence, $\tilde{x}$ cannot be high.
Every $1$-generic is diagonally computable \cite[Theorem 2.24.5]{DHBook}, which is equivalent to being non-autocomplex \cite[Theorem 8.16.4]{DHBook}.
By \cite[Theorem 8.16.8]{DHBook}, this implies that $\tilde{x}$ is computably i.o.\ traceable, that is, for any $f\leq_T\tilde{x}$, there is a computable sequence $(T_r)_{r<n}$ of (canonical indices of) finite sets such that $f(r)\in T_r$ and $|T_r|\leq r$ for infinitely many $r\in\om$.
The rest of the proof is just a few modification of the known fact that computable i.o.\ traceability implies effective Hausdorff nullness w.r.t.\ all computable gauge functions, cf.\ \cite{KiMi15}.

Let $p\leq_T\tilde{x}$ be a Cauchy name of $x$.
Then, there is a $p$-computable sequence $q=(q_r)_{r\in\om}$ of $n$-tuples of rationals such that $d(x,q_r)<2^{-r}$.
Since $q\leq_T\tilde{x}$, there is a computable sequence $(T_r)_{r\in\om}$ of finite sets such that $q_r\in T_r$ and $|T_r|\leq r$ for almost all $r\in\om$.
By Kolmogorov's lemma (cf.\ \cite[Theorem 3.2.2]{DHBook}), one can construct a computable measure machine $M$ such that $C_M(q_r)\leq \log|T_r|+\log r\leq 2\log r$ for infinitely many $r\in\om$.
Therefore,
\[\dim_H^{\rm Sch}(x)\leq \liminf_{r\to\infty}\frac{K_{M,r}(x)}{r}\leq\liminf_{r\to\infty}\frac{K_M(q_r)}{r}\leq\liminf_{r\to\infty}\frac{2\log r}{r}=0.\]

Consequently, we have $0=\dim_H^{\rm Sch}(x)<\dim(x)=n$ as desired.
\end{proof}

We say that $x\in\repsp{X}$ is {\em c.o.\ $s$-dimensional}, written $\kdim(x)\leq s$, if $x$ is contained in a $\Pi^0_1$ subset of $\mathcal{X}$ of Hausdorff dimension $\leq s$.

\begin{obs}\label{obs:basic-inequality-Kurtz-Haus}
Let $x$ be a point in a computable compactum.
Then $\dim(x)\leq\kdim(x)$ holds.
\end{obs}


\begin{proof}
Since the topological dimension $\dim(P)$ is smaller than or equal to the Hausdorff dimension $\dim_H(P)$, every c.o.\ $s$-dimensional point is $\lfloor s\rfloor$-dimensional.
\end{proof}

\begin{prop}
For any $n\geq 1$, there is $x\in[0,1]^n$ such that $0=\dim_P(x)<\dim(x)=n$.
\end{prop}

\begin{proof}
It is known that every noncomputable c.e.\ set $B\subseteq\om$ computes a weakly $1$-generic real $\alpha\in 2^\om$ (cf.\ \cite[Proposition 2.24.2]{DHBook}).
This $\alpha$ can be written as $\bigoplus_{i<n}\alpha_i$.
Note that $x=(0.\alpha_i)_{i<n}$ is weakly $1$-generic in $[0,1]^n$ since if $P\subseteq[0,1]^n$ is nowhere dense $\Pi^0_1$ then so is $\{\beta\in 2^\om:(0.\beta_i)_{i\in\om}\in P\}$.
Hence, by Observation \ref{obs:proper-n-dim-w1g}, $\dim(x)=n$.
Now let $B$ be a noncomputable c.e.\ set of array computable degree.
By \cite[Theorem 2.23.13]{DHBook}, $B$ is c.e.\ traceable, that is, for any $f\leq_TB$, there is a computable sequence $(T_r)_{r\in\om}$ of c.e.\ sets such that $f(r)\in T_r$ and $|T_r|\leq r$ for almost all $r\in\om$.
The rest of the proof is just a few modification of the known fact that c.e.\ traceability implies effective packing nullness w.r.t.\ all computable gauge functions, cf.\ \cite{KiMi15}.

Let $p\leq_TB$ be a Cauchy name of $x$.
Then, there is a $p$-computable sequence $q=(q_r)_{r\in\om}$ of $n$-tuples of rationals such that $d(x,q_r)<2^{-r}$.
Since $q\leq_TB$, there is a computable sequence $(T_r)_{r\in\om}$ of c.e.\ sets such that $q_r\in T_r$ and $|T_r|\leq r$ for almost all $r\in\om$.
By Kolmogorov's lemma (cf.\ \cite[Theorem 3.2.2]{DHBook}) as in the proof of Proposition \ref{prop:2-gen-Sch-dim}, we get
\[\dim_P(x)=\limsup_{r\to\infty}\frac{K_r(x)}{r}\leq\limsup_{r\to\infty}\frac{K(q_r)}{r}\leq\limsup_{r\to\infty}\frac{2\log r}{r}=0.\]

Consequently, we have $0=\dim_P(x)<\dim(x)=n$ as desired.
\end{proof}

\subsubsection{Box counting dimension}

It is known that the effective box-counting dimension of a point in $2^\om$ is equivalent to its effective packing dimension (cf.\ \cite[Section 13.11.4]{DHBook}).
One can also show a similar result for Schnorr dimensions in Euclidean spaces.

\begin{lemma}\label{lem:box-counting-Kolmogorov-complexity}
The following are equivalent:
\begin{enumerate}
\item $x$ is contained in a $\Pi^0_1$ set of upper box-counting dimension $<s$.
\item $\dim_P^{\rm Sch}(x)<s$, that is, there is a compression algorithm $M$ such that
\[\limsup_{r\to\infty}\frac{C_{M,r}(x)}{r}<s.\]
\end{enumerate}
\end{lemma}

\begin{proof}
Let $E$ be a $\Pi^0_1$ set such that $x\in E$ and $\ubdim(E)<s$.
Then there is a sufficiently small rational $q>0$ such that $|E|_r<r^{-s}$ whenever $0<r<q$.
By computable compactness of $E$ and by Observation \ref{obs:box-dimension-rational}, given a rational $r<q$, one can effectively find an open cover of $E$ witnessing the above inequality.
By Kolmogorov's lemma (cf.\ \cite[Theorem 3.2.2]{DHBook}), we get a decidable machine $M$ that $C_{M,r}(x)<\log(2^{-rs})+C(r)+O(1)\leq sr+\log(r)+O(1)$ whenever $r<q$.
This implies the desired inequality.

Conversely, let $q$ be such that any positive rational $r<q$ satisfies the inequality in (2) for a decidable machine $M$.
Consider $D_n=\{\overline{B}_{r_n}(p):p\in\mathbb{Q}^n\mbox{ and }C_M(p)<s\log(r_n^{-1})\}$.
Then $|D_n|\leq 2^{s\log(r_n^{-1})}=r_n^{-s}$.
Clearly $\bigcap_nD_n$ is $\Pi^0_1$.
\end{proof}

\begin{theorem}\label{thm:main-Kol-complexity}
The following are equivalent for $x\in[0,1]^\om$:
\begin{enumerate}
\item $x$ is $n$-dimensional.
\item There is $y\equiv_{tt}x$ such that $\kdim(y)<n+1$.
\item There is $y\equiv_{tt}x$ such that for any $\ep>0$, $\kdim(y)<n+\ep$.
\item There is $y\equiv_{tt}x$ such that for any $\ep>0$, $\dim_P^{\rm Sch}(x)<n+\ep$, that is, there is a compression algorithm $M$ such that
\[\limsup_{r\to\infty}\frac{C_{M,r}(y)}{r}<n+\ep.\]
\end{enumerate}
\end{theorem}

\begin{proof}
(4)$\Rightarrow$(3)$\Rightarrow$(2): Obvious.
(2)$\Rightarrow$(1):
We show the contrapositive.
If $x$ is not $n$-dimensional, then $\dim(x)\geq n+1$.
Assume that $y\equiv_{tt}x$ is given.
Then $\dim(y)\geq n+1$.
Therefore, $\kdim(y)\geq n+1$ by Observation \ref{obs:basic-inequality-Kurtz-Haus}.
(1) $\Rightarrow$ (4):
Since $x$ is $n$-dimensional, there is an $n$-dimensional $\Pi^0_1$ set $P\subseteq[0,1]^\om$ containing $x$.
By Corollary \ref{cor:Assouad-dimension}, $P$ is computably embedded into a computable compactum $M\subseteq\mathbb{R}^{2n+1}$ of Assouad dimension $n$.
Let $Q$ be the embedded image, which is $\Pi^0_1$ since $Q$ is the image of a computable function $p$ on a computably compact set $P$.
Then $\ubdim(Q)\leq\dim_A(Q)\leq\dim_A(M)=n$.
By Lemma \ref{lem:main-lemma}, $y=p(x)\in Q$ is $tt$-equivalent to $x$.
By Lemma \ref{lem:box-counting-Kolmogorov-complexity}, this $y$ satisfies the desired inequality.
\end{proof}

It is not known whether we can remove $\ep$ from the above characterization.

\begin{question}
If $x$ is $n$-dimensional, does there exist $y\equiv_{tt}x$ such that $\kdim(y)=\dim_{P}^{\rm Sch}(y)=n$?
\end{question}

\section{Degree structures}

\subsection{Pseudo-arc}

By Fact \ref{fact:MR-computable-arc}, a point is contained in a computable planar arc iff it has a $tt$-degree of a point in $\mathbb{R}$.
In particular, if $A\subseteq\mathbb{R}^2$ is a nontrivial computable arc, then its $tt$-degree structure $\mathcal{D}_{tt}(A)$ is equal to $\mathcal{D}_{tt}(\mathbb{R})$.
In this section, we give an example of a computable {\em arc-like} continuum whose $tt$-degree structure is very different from $\mathcal{D}_{tt}(\mathbb{R})$.

By a {\em compactum} we mean a compact metric space, and by a {\em continuum} we mean a connected compactum.
A continuum is {\em nondegenerated} if it has at least two points.
A continuum is {\em hereditarily indecomposable} if one of two given nondegenerated continua is included in the other.
A continuum is {\em arc-like} (or {\em chainable}) if it is the inverse limit of a sequence of arcs (equivalently, it has a chain-open cover of arbitrarily small mesh).
Every arc-like continuum is one-dimensional.
A {\em pseudo-arc} is a hereditarily indecomposable arc-like continuum.

It is easy to check that the standard zig-zag construction of a pseudo-arc gives a computable presentation within $\mathbb{R}^2$.

\begin{obs}
There is a computable planar pseudo-arc.
\end{obs}

We show that the $tt$-degree structures of an arc and a pseudo-arc form a ``{\em minimal pair}'' in the following sense:

\begin{prop}\label{prop:pseudo-arc}
The $tt$-degree structure of a computable pseudo-arc $\repsp{A}$ is incomparable with that of $\mathbb{R}$.
Moreover, 
\[\mathcal{D}_{tt}(\repsp{A})\cap\mathcal{D}_{tt}(\mathbb{R})=\mathcal{D}_{tt}(2^\om).\]
\end{prop}

\begin{proof}
It is well-known that a perfect computable compactum contains a computable copy of Cantor space, cf.\ \cite[Exercise 3D.15]{MosBook}.
Thereofore, $\mathcal{D}_{tt}(2^\om)\subseteq\mathcal{D}_{tt}(\repsp{A})$ since $\repsp{A}$ is a perfect computable compactum.
Indeed, $\mathcal{D}_{tt}(2^\om)\subsetneq\mathcal{D}_{tt}(\repsp{A})$ since the topological dimension is first-level invariant (Corollary \ref{cor:dimension-is-invariant}), and $\repsp{A}$ is one-dimensional while $2^\om$ is zero-dimensional.

To verify the second assertion, given $x\in \repsp{A}$ and $y\in\mathbb{R}$, assume that $x\equiv_{tt}y$.
By Lemma \ref{lem:main-lemma}, there are $\Pi^0_1$ sets $D\subseteq \repsp{A}$ and $E\subseteq\mathbb{R}$ such that $x\in D$, $y\in E$, and $D$ is homeomorphic to $E$.
If $D$ is zero-dimensional, by Theorem \ref{thm:universal-menger}, $x$ and $y$ have $2^\om$-$tt$-degrees.
Otherwise, $D$ (and hence $E$) contains a nondegenerated continuum (since a compactum is zero-dimensional iff it is punctiform, cf.\ \cite[Theorem 1.4.5]{EngBook}), and every nondegenerated subcontinuum of $D\subseteq\repsp{A}$ is hereditarily indecomposable.
However, this means that $E$ has a hereditarily indecomposable continuum as a subspace, and thus $\mathbb{R}$ contains a hereditarily indecomposable continuum, which is impossible.
%
\end{proof}


A similar argument shows that there is a pair of $(n+1)$-dimensional compacta which have no common proper $(n+1)$-dimensional uniform degrees.

\begin{prop}
For every $n$, there is a computable $(n+1)$-dimensional continuum $\repsp{B}_{n+1}$ such that the common $tt$-degrees of $\mathbb{R}^{n+1}$ and $\repsp{B}_{n+1}$ are only $n$-dimensional ones.
\end{prop}

\begin{proof}
Let $\repsp{B}_{n+1}$ be a hereditarily indecomposable $(n+1)$-dimensional continuum.
It is easy to check that the construction in van Mill \cite[Corollary 3.8.3]{vMBook} is effective.
Thus, such $\repsp{B}_{n+1}$ can be computable.
Suppose for the sake of contradiction that there is a proper $n$-dimensional point $x\in\repsp{B}_{n+1}$ which is $tt$-equivalent to a point in $\mathbb{R}^{n+1}$.
Then, by Lemma \ref{lem:main-lemma}, there is a $\Pi^0_1$ set $P\subseteq\repsp{B}_{n+1}$ with $x\in P$ is computably embedded into $\mathbb{R}^{n+1}$.
However, $P$ must be $(n+1)$-dimensional since $x$ is not $n$-dimensional, and $x\in P$.
Thus, $P$ contains an $(n+1)$-dimensional continuum $C$ (cf.\ \cite[Theorem VI.8]{HWBook}).
Since $\repsp{B}_{n+1}$ is hereditarily indecomposable, so is $C$.
Thus, the embedded image of $C$ in $\mathbb{R}^{n+1}$ is a hereditarily indecomposable $(n+1)$-dimensional continuum.
However, a subset of $\mathbb{R}^{n+1}$ is $(n+1)$-dimensional if and only if it contains a homeomorphic copy of $\mathbb{R}^{n+1}$ (see \cite[Theorem IV.3]{HWBook}), and $\mathbb{R}^{n+1}$ clearly contains a decomposable continuum, a contradiction.
Consequently, if $x\in\repsp{B}_{n+1}$ is $tt$-equivalent to a point in $\mathbb{R}^n$, then $x$ has to be $n$-dimensional.
\end{proof}



%

\subsection{Arc-like continua}

Recall that an arc-like continuum is an inverse limit of arcs.
In this section, we discuss a technique for studying the $tt$-degrees of points in simple inverse limits of arcs.
For instance, consider the piecewise linear function $f:\interval\to\interval$ defined by
\[
f(x)=
\begin{cases}
2x&\mbox{ if }x\leq 1/2,\\
2-2x&\mbox{ if }x\geq 1/2.
\end{cases}
\]
That is, $f$ is a tent map.
For the inverse system $(I_n,f_n)$, where $I_n=\interval$ and $f_n=f$, the inverse limit $K=\varprojlim (I_n,f_n)$ is known as {\em Knaster's bucket handle}.

Assume that a continuous function $f\colon\interval\to\interval$ is given.
A point $x\in\interval$ is {\em preperiodic} (a.k.a.\ {\em eventually periodic}) if there is $n$ such that $f^n(x)$ is periodic, that is, the (forward) orbit of $x$ is finite.
A point $x\in\interval$ is {\em asymptotically periodic} if the $\om$-limit set of $x$ (that is, the set of cluster points of the orbit of $x$) is a periodic orbit.
It is clear that every preperiodic point is asymptotically periodic.
Moreover, note that the closure of the orbit $O$ of an asymptotically periodic point is the union of $O$ and a finite orbit.
We say that a point $x\in I$ is {\em effectively asymptotically periodic} (or {\em e.a.\ periodic}) if the closure $\overline{O}$ of the orbit $O$ of $x$ is $\Pi^0_1$, and if $\overline{O}$ is the union of $O$ and a finite orbit.
Clearly, every preperiodic point is e.a.\ periodic, and every e.a.\ periodic point has a $\Delta^0_2$-orbit.

For a function $f:J\to K$, where $J$ and $K$ are closed subsets of the unit interval $\mathbb{I}$, let $\locex(f)\subseteq J$ be the set of all local extrema of $f$ except for end points.
Hereafter, by a local extremum we mean a point in $\locex(f)$, and by a local extremum value we mean a point in $\locexv(f):=f[\locex(f)]$.
For instance, if $f$ is a tent map, then $\locex=\{1/2\}$, and $\locexv=\{1\}$.

\begin{example}
The tent map is piecewise linear, finite-to-one, computable function, all of whose local extrema are preperiodic, but not periodic.
\end{example}

\begin{example}
There is a finite-to-one function which has a non-preperiodic local extremum, but each of whose local extremum is e.a.\ periodic.
For instance, consider the piecewise linear function whose graph is the union of five line segments connecting six points $(0,0)$, $(1/5,1/6)$, $(2/5,4/5)$, $(3/5,1/5)$, $(4/5,5/6)$, and $(1,1)$.
The orbit $O$ of a local extremum $x$ approaches to either $0$ or $1$, but it is not necessarily finite.
\end{example}

We see that, if a computable arc-like continuum is constructed by a simple inverse limit, then it contains no more than $(\mathbb{R}\times 2^\om)$-$tt$-degrees.

\begin{theorem}\label{thm:Kbuc-han}
Let $K$ be an inverse limit of arcs with a single bonding map $f$, where $f$ is a piecewise monotone, finite-to-one, computable function, all of whose local extrema are e.a.\ periodic.
Then, every point in $K$ is $tt$-equivalent to a point in $\mathbb{R}\times 2^\om$, that is, $\mathcal{D}_{tt}(K)\subseteq\mathcal{D}_{tt}(\mathbb{R}\times 2^\om)$.
\end{theorem}

To prove Theorem \ref{thm:Kbuc-han} we need to examine the property of the orbit of a local extremum.
Note that the inverse limit $K$ is of the following form:
\[K=\{x\in[0,1]^\om:(\forall n\in\om)\;f(x(n+1))=x(n)\}\]

Let $\locex(K)$ be the set of all points $x\in K$ such that $x(n)$ is a local extremum value of $f$ at some $n$, that is,
\[\locex(K)=\{x\in K:(\exists n)\;x(n)\in \locexv(f)\}.\]
This is always $F_\sigma$ whenever $f$ is continuous.
For $x\in K$, consider
\begin{align*}
\locext(x)&=\{n\in\om:x(n)\in \locexv(f)\},\\
[x]_{\locext}&=\{y\in K:\locext_f(x)=\locext_f(y)\}.
\end{align*}



\begin{lemma}\label{lem:Ex-effective}
Assume that $K$ satisfies the assumption in Theorem \ref{thm:Kbuc-han}.
Then $\locex(K)$ is $\Delta^0_2$, $\locext(x)$ is computable, and $[x]_{\locext}$ is $\Sigma^0_2$ for any $x\in K$.
\end{lemma}
 
\begin{proof}
We first show that $\locex(K)$ is $\Delta^0_2$.
Note that $f$ has only finitely many local extrema $z$.
Since $z$ is e.a.\ preriodic, the orbit $O_z$ of $z$ is $\Delta^0_2$.
If $z$ is preperiodic, since $O_z$ is finite, and $f$ is finite-to-one, $(f^n)^{-1}[O_z]$ is finite.
Thus, uniformly in $n$, one can find a c.e.\ open set $U_n^z$ in $I=\interval$ such that $U_n^z\cap(f^n)^{-1}[O_z]=\{f(z)\}$.
Next assume that $z$ is not preperiodic.
Note that $f^n(z)$ is not contained in the closure of $f^{n+1}[O_z]$.
Otherwise, either $f^n(z)=f^m(z)$ for some $m>n$ or $f^n(z)\in\overline{O_z}\setminus O_z$.
The former means that $f^n(z)$ is periodic, and therefore $z$ is preperiodic.
The latter means that $f^n(z)$ is contained in a finite orbit since $z$ is asymptotically periodic.
Thus, this implies that $z$ is preperiodic, which contradicts our assumption.
The above argument also shows that the closure of $f^{n+1}[O_z]$ is $\Pi^0_1$ since it is the difference of $\overline{O_z}$ and $\{z,f(z),\dots, f^n(z)\}$, and $f^j(z)$ is isolated in $\overline{O_z}$ for any $j$.
Now, there is an open neighborhood $V$ of $f^n(z)$ such that $V$ does not intersect with the closure of $f^{n+1}[O_z]$.
Therefore, given a computable sequence of open balls $(B_i)_{i\in\om}$ such that $\overline{f^{n+1}[O_z]}=I\setminus\bigcup_iB_i$, there is $j$ such that $f^n(z)\in B_j$.
One can effectively find such a $B_j$, and then define $V_n=(f^{n-1})^{-1}[B_j]$.
Then, note that
\[f(z)\in V_n\cap (f^{n-1})^{-1}[O_z]\subseteq(f^{n-1})^{-1}\{z,f(z),f^2(z),\dots,f^n(z)\}.\]
The latter set is finite since $f$ is finite-to-one.
Thus, one can effectively find an open set $U^z_{n-1}\subseteq V_n$ such that $U^z_{n-1}\cap (f^{n-1})^{-1}[O_z]=\{f(z)\}$.

Consequently, $x\in\locex(K)$ if and only if there is $z\in\locex(f)$ such that $x(0)\in O_z$ and $x(n)\in U^z_n$ for some $n\in\om$.
This gives a $\Delta^0_2$ definition of $\locex(K)$ since $\locex(f)$ is finite, and $O_z$ is $\Delta^0_2$.

If $x(n)$ is a local extremum value $z$ for some $n$, but no such $z$ is periodic, then $\locext(x)$ must be finite, since $f$ has only finitely many local extrema $z$.
In this case, it is clear that $[x]_{\locext}$ is $\Sigma^0_2$.
Therefore, if $\locext(x)$ is infinite, then $x(n)$ is a periodic local extremum value $z$ for infinitely many $n\in\om$.
This means that $(x(n))_{n\in\om}$ repeats a finite sequence $\sigma(z)$ determined by $z$ forever.
In particular, $\locext(x)$ is computable.

To estimate the complexity of $[x]_{\locext}$, let $p$ be the period of $z$.
Then, if $x$ attains $z$ infinitely often, then there are only $p$ many candidate for such $x$, that is, there is $m<p$ such that
\[\langle x(n):kp+m\leq n<(k+1)p+m\rangle=\sigma(z)\]
for any $k$.
Now, if $\locext(x)$ is infinite, then $x$ has to attain some periodic local extremum value $z$ infinitely often, but there are only finitely many such $z$.
This means that $[x]_{\locext}$ is a finite union of finite sets of computable points.
Consequently, $[x]_{\locext}$ is $\Sigma^0_2$.
\end{proof}

\begin{remark}
A point $x\in I$ is {\em recurrent} if for any open neighborhood $U$ of $x$, $f^m(x)\in U$ for some $m>0$.
We say that $x\in I$ is {\em prerecurrent} if there is $n$ such that for any open neighborhood $U$ of $x$, $f^m(x)\in f^{n}[U]$ for some $m>n$.
Clearly, every periodic point is recurrent, and every preperiodic point is prerecurrent.
The proof of Lemma \ref{lem:Ex-effective} shows that an asymptotically periodic point is preperiodic if and only if it is prerecurrent.
\end{remark}

We will prove Theorem \ref{thm:Kbuc-han} in a slightly more general form to include an inverse system of many bonding maps $f_n:I_{n+1}\to I_n$.
We now say that an inverse system $(I_n,f_n)$ is {\em effectively basic} if
\begin{enumerate}
\item $I_n$ is a finite union of subintervals of $\interval$.
\item $f_n$ is a piecewise monotone, finite-to-one, function.
\item $(I_n,f_n)$ be uniformly computable.
\item Given $n$, one can effectively enumerate all local extreme values of $f_n$ w.r.t.\ the usual ordering $<$ on reals without repetition.
\end{enumerate}
Note that the condition (2) ensures that $\locex(f_n)$ is finite.
Then we consider:
\begin{align*}
\locex(K)&=\{x\in K:(\exists n)\;x(n)\in\locexv(f_n)]\},\\
\locext(x)&=\{n\in\om:x(n)\in \locexv(f_n)]\},\\
[x]_{\locext}&=\{y\in K:\locext(x)=\locext(y)\}.
\end{align*}

Most natural inverse limits with unimodal bonding maps also satisfy this property.
Now we prove Theorem \ref{thm:Kbuc-han}.

\begin{lemma}\label{lem:Ex-eff-uniform}
Let $K$ be an inverse limit of an effectively basic inverse system, which satisfies the conclusion in Lemma \ref{lem:Ex-effective}.
Then, every point in $K$ is $tt$-equivalent to a point in $\mathbb{R}\times 2^\om$, that is, $\mathcal{D}_{tt}(K)\subseteq\mathcal{D}_{tt}(\mathbb{R}\times 2^\om)$.
\end{lemma}

\begin{proof}
We will code each $x\not\in\locex(K)$ as $(x(0),c_x)\in I_0\times \om^{<\om}$, where $c_x$ is computably dominated, and therefore, has a $2^\om$-$tt$-degree.
Given $n$, the increasing enumeration $(a_i)_{i<b}$ of $\locexv(f_n)$ divides $I_n$ into finitely many subintervals $I_n\cap[a_{i-1},a_{i}]$, where $a_{-1}=-0.1$ and $a_b=1.1$.
Then, for any $y,z\in (a_{i-1},a_i)$, the level sets $f_n^{-1}\{y\}$ and $f_n^{-1}\{z\}$ have the same cardinalities, say $\ell(i)$, since $f_n$ does not attain a local extreme value in this open interval.
One can effectively find a computable function $g_{n,i}:(a_{i-1},a_i)\to I_{n+1}^{\ell(i)}$ such that $g_{n,i}(y)$ enumerates $f_n^{-1}\{y\}$ w.r.t.\ the usual ordering $<$ on reals without repetition.
Now, we can effectively find a unique $i\leq b$ such that $x(n)\in (a_{i-1},a_i)$, and then we define $c_x(n)=k$ if and only if $x(n+1)$ is the $k$-th element of $f_n^{-1}\{x(n)\}$, that is, $g_{n,i}(x(n))(k)$.
By effectivity of $g_{n,i}$ uniformly in $n$ and $i$, and by our assumption that $\locex(K)$ is $\Delta^0_2$ (hence, the complement of $\locex(K)$ is a countable union of $\Pi^0_1$ sets), the code $c_x$ is $tt$-reducible to $x$.

Put $E_0=\locexv(f_0)$.
We claim that, indeed, this argument gives a computable function $s\mapsto T(s)$ such that $x\mapsto(x(0),c_x)$ is a computable homeomorphism between $K\setminus{\rm Ex}(K)$ and $\{(s,t):s\in I_0\setminus E_0\mbox{ and }t\in [T(s)]\}$ where $T(s)\subseteq\om^{<\om}$ is a computably bounded finite branching computable tree.
Assume that $(x(0),\sigma)$, where $\sigma\in\om^n$, is a code of a finite sequence $(x(k))_{k\leq n}$.
Then, as mentioned above, we can compute $i$ such that $x(k)\in  (a_{i-1},a_i)$, and thus get $\ell(i)$ and $g_{n,i}$ in an effective manner.
We declare that $T(x(0))$ has $\ell(i)$ many immediate successors of $\sigma$, and $(x(0),\sigma\fr k)$ codes $g_{n,i}(x(n))(k)$.
Clearly, the decoding procedure is effective, and defined on $\{x(0)\}\times [T(x(0))]$.
This verifies the claim.
Consequently, $x$ is $tt$-equivalent to $(x(0),c_x)$, that is, $x$ has an $(\mathbb{R}\times 2^\om)$-uniform degree.

We next consider $x\in\locex(K)$.
By our assumption, $A:=\locext(x)$ is computable, and $[x]_{\locext}$ is $\Sigma^0_2$.
If $n\not\in A$, we have $x(n)\not\in\locexv(f_n)$, and then we define $c_x(n)$, which codes $x(n+1)$, as before.
If $n\in A$, we have $x(n)\in \locexv(f_n)=\{a_i\}_{i<b}$.
Since $\{a_i\}_{i<b}$ is discrete, one can compute $i<b$ such that $x(n)=i$.
We then effectively enumerate $f_n^{-1}\{a_i\}$, and define $c_x(n)=k$ if and only if $x(n+1)$ is the $k$-th element of $f_n^{-1}\{a_i\}$.
By the same argument as above, and by the assumption that $[x]_{\locext}=\{y:\locext(y)=A\}$ is $\Sigma^0_2$ (hence, a countable union of $\Pi^0_1$ sets), this shows that $x$ is $tt$-equivalent to $(x(0),c_x)$.
\end{proof}

Theorem \ref{thm:Kbuc-han} follows from Lemmas \ref{lem:Ex-effective} and \ref{lem:Ex-eff-uniform}.
%
We are now interested in how many $tt$-degree structures of computable arc-like continua there exist.

\begin{obs}\label{obs:arc-like-example}
The $tt$-degree structures of the following spaces are realized by those of computable arc-like planar continua:
\[\mathbb{R},\;\mathbb{R}\times 2^\om,\;\repsp{A},\;\repsp{A}\cup\mathbb{R}\mbox{, and }\repsp{A}\cup(\mathbb{R}\times 2^\om),\]
where $\repsp{A}$ is a computable pseudo-arc.
\end{obs}

\begin{proof}
First claim that the $tt$-degree structure of Knaster's bucket handle $K$ is that of $\mathbb{R}\times 2^\om$.
This follows from a simple observation that the set $\{x\in K:x(0)\in[1/8,1/4]\}$ is homeomorphic to $\interval\times 2^\om$ because for any $x\in[1/8,1/4]$, it is easy to see that $x\not\in\locex(K)$, and the tree $T(x)$ in Lemma \ref{lem:Ex-eff-uniform} is exactly the full binary tree $2^{<\om}$.
Then  Theorem \ref{thm:Kbuc-han} implies that the $tt$-degree structure of $K$ is the same as that of $\mathbb{R}\times 2^\om$.

Next, let $B$ be the result by connecting $[0,1]$ at an end point of a computable pseudo-arc $\repsp{A}$.
It is clear that $B$ is a computable arc-like continuum, and its $tt$-degree structure is that of $\repsp{A}\cup\mathbb{R}$.
In a similar manner, by connecting Knaster's bucket handle and the pseudo-arc, the $tt$-degree structure $\repsp{A}\cup(\mathbb{R}\times 2^\om)$ can be realized as the $tt$-degree structure of a computable arc-like continuum.
\end{proof}

We do not know any other example even if we allow a space to be any computable circle-like, or tree-like, non-planar continuum.
For instance, consider a circle-like continuum, known as a solenoid.
For a prime number $p$, define $f_p:\mathbb{T}\to\mathbb{T}$ by $f_p(z)=z^p$, where $\mathbb{T}$ is the unit circle in the complex plane $\mathbb{C}$.
For a sequence of prime numbers $P=(p_{i})_{i\in\om}$, the {\em $P$-solenoid} is defined by $S_P=\varprojlim(\mathbb{T},f_{p_i})$.
Note that if $P$ is a computable sequence, the $P$-solenoid is computable. 
It is known that if $P$ and $Q$ are sufficiently different, $S_P$ and $S_Q$ are not homeomorphic.
However, it is not hard to check that all computable solenoids have the same $tt$-degree structures, that of $\mathbb{R}\times 2^\om$.

\begin{obs}
There is a computable one-dimensional planar continuum whose $tt$-degree structure is different from those in Observation \ref{obs:arc-like-example}.
\end{obs}

\begin{proof}
Note that a universal planar curve $C=M^2_1(\mathbf{3})$ (a.k.a.\ Sierpinsk\'i's carpet) contains both a pseudo-arc $\repsp{A}$ and Knaster's bucket handle $K$.
Typically, such a universal space can be computably embedded into any nonempty open subset of the space.
This implies that the $tt$-degree structure of $C$ is {\em join-irreducible} in the sense that if the $tt$-degree structures of spaces $A$ and $B$ are strictly smaller than that of $C$, then so is the $tt$-degree structure of the disjoint sum $A\sqcup B$.
This is because if a weakly $1$-generic point $x\in C$ is $tt$-equivalent to a point $y\in A\sqcup B$ (so $y\in A$ or $y\in B$), then by Lemma \ref{lem:main-lemma}, there is a $\Pi^0_1$ set $P\subseteq C$ with $x\in C$ is computably embedded into $A\sqcup B$.
Let $f$ be such a computable embedding.
Then, both $P\cap f^{-1}[A]$ and $P\cap f^{-1}[B]$ are $\Pi^0_1$ since $A$ and $B$ are clopen in $A\sqcup B$.
Either $P\cap f^{-1}[A]$ or $P\cap f^{-1}[B]$ contains $x$, and therefore contains a nonempty interior.
As mentioned before, any nonempty open subset of $C$ contains a computable copy of $C$.
Hence, $C$ is computably embedded into either $A$ or $B$, which implies that the $tt$-degree structure of $C$ is included in either $A$ or $B$.
This verifies join-irreducibility of $C$.
Consequently, $C$ has more $tt$-degrees than $\repsp{A}\sqcup K$.
\end{proof}

\subsection{Condensation of singularities}

In this section, we investigate the order-theoretic property of a collection of $tt$-degree structures.
More precisely, for a collection $\mathcal{S}$ of computable metric spaces, consider the following ordered structure:
\[\mathfrak{D}_{tt}(\mathcal{S})=(\{\mathcal{D}_{tt}(\repsp{X}):\repsp{X}\in\mathcal{S}\},\subseteq).\]
This set forms a countable upper semilattice (the join is given as the $tt$-degree structure of the disjoint sum $\repsp{X}\sqcup\repsp{Y}$).
Let $\mathcal{K}_n$ be the collection of all $n$-dimensional computable compacta, and define $\mathfrak{D}_{tt}(n\text{-dim})=\mathfrak{D}_{tt}(\mathcal{K}_n)$, the collection of the $tt$-degree structures of $n$-dimensional computable compacta.
For $n>1$, we show that this is a universal countable upper semilattice.

\begin{theorem}\label{thm:embedding-usl}
For any $n>1$, any countable upper semilattice can be embedded into $\mathfrak{D}_{tt}(n\text{-}\mathrm{dim})$.
\end{theorem}


One can obtain a more effective content.
As seen in Propositions \ref{prop:hypim-free-str} and \ref{prop:nontotal-T-vs-tt}, if $x$ is either hyperimmune-free or non-total, then
\[(\forall y)\;[\dim(x)<\infty\;\Longrightarrow\;(y\leq_Tx\iff y\leq_{tt}x).\]

Thus, it is natural to ask what happens if $x$ is hyperimmune and total.
We will see that Theorem \ref{thm:embedding-usl} holds inside any hyperimmune total $T$-degree.
To be precise, consider the following notions for any $T$-degree $\mathbf{d}$:
\begin{align*}
\mathcal{D}^\mathbf{d}_{tt}(\repsp{X})&=\{\deg_{tt}(x):x\in\repsp{X}\mbox{ and }\deg_T(x)=\mathbf{d}\},\\
\mathfrak{D}^\mathbf{d}_{tt}(\mathcal{S})&=(\{\mathcal{D}^\mathbf{d}_{tt}(\repsp{X}):\repsp{X}\in\mathcal{S}\},\subseteq).
\end{align*}

Moreover, let $\mathcal{C}_n$ be the collection of all $n$-dimensional computable continua, and define $\mathfrak{D}_{tt}^\mathbf{d}(n\text{-}\mathrm{dim},\mathrm{cont})=\mathfrak{D}^\mathbf{d}_{tt}(\mathcal{C}_n)$.

\begin{example}
If $\mathbf{d}$ is hyperimmune-free and total, then $\mathfrak{D}^\mathbf{d}_{tt}(n\mbox{-}\dim)$ is a singleton.
\end{example}

We claim that any countable upper semilattice embeds into $\mathfrak{D}^\mathbf{d}_{tt}(n\mbox{-}\dim)$ whenever $\mathbf{d}$ is hyperimmune and total.



\begin{theorem}\label{thm:usl-embedding2}
Let $\mathbf{d}$ be a hyperimmune total $T$-degree, and $n>0$.
Then, any countable upper semilattice can be embedded into $\mathfrak{D}_{tt}^\mathbf{d}(n\text{-}\mathrm{dim},\mathrm{cont})$.
\end{theorem}

To prove this, we need a straightforward effectivization of Chatyrko-Pol's construction \cite{ChPo} which utilizes the method of {\em condensation of singularities}.
Our argument below is merely a careful analysis of the construction in \cite{ChPo}.

\subsubsection{The space $S(E,K,t)$.}
Let $E$ and $K$ be computable continua in $[0,1]^\om$.
Fix a computable dense subset $\{a_i\}_{i\in\om}\subseteq K$.
Then let $L_i\subseteq[0,1]^\om\times[2^{-(i+1)},2^{-i}]$ be the line segment from $(a_i,2^{-i})$ to $(a_{i+1},2^{-(i+1)})$.
Fix a computable injective parametrization $h:[0,\infty)\to L:=\bigcup_iL_i$ such that $h(i)=(a_i,2^{-i})$.
Given a computable point $t\in\om$, define $f:E\setminus\{t\}\to L$ as the following computable function:
\[f(x)=h(d(x,t)^{-1}).\]
Then, we define $S(E,K,t)\subseteq E\times\overline{L}$ as the closure of the graph of $f$ in $E\times[0,1]^\om\times[0,1]$, that is,
\[S(E,K,t)={\rm Graph}(f)\cup(\{t\}\times K\times\{0\}).\]
Note that $S(E,K,t)$ is still a computable continuum.
Define a computable embedding $f^\ast:E\setminus\{t\}\hookrightarrow S(E,K,t)$ as follows:
\[f^\ast(x)=(x,f(x)).\]
Given a dense subset $\{q_i\}_{i\in\om}\subseteq E\setminus\{t\}$, clearly, $\{f^\ast(q_i)\}_{i\in\om}$ forms a dense subset of $S(E,K,t)$.
Let $p:S(E,K,t)\to E$ be the projection.
Note that
\[p^{-1}(t)=\{t\}\times K\times\{0\},\qquad p^{-1}(x)=\{f^\ast(x)\}\mbox{ for $x\not=t$}.\]
This has the following property (see \cite[Lemma 2.1]{ChPo}):
\begin{enumerate}
\item The fiber $p^{-1}(t)$ is a copy of $K$, and the other fibers are singletons.
\item If $L$ is a continuum in $S(E,K,t)$ such that $p[L]$ is nondegenerated and it contains $t$, then $p^{-1}(t)\subseteq L$.
\end{enumerate}

For instance, $S([0,1],[0,1],0)$ looks quite similar to topologist's sine curve.

\subsubsection{The space $S(E,K,Q)$.}
We now iterate this procedure.
Fix a computable dense subset $Q=\{q_i\}_{i\in\om}\subseteq E$.
Begin with $E_0=E$ and $f_0^\ast={\rm id}$, and first consider $E_1=S(E_0,K,q_0)$.
Then, we get a computable embedding $f^\ast_1:E\setminus\{q_0\}\hookrightarrow E_1$ as above.
Assume that we have constructed $E_n$ and $f^\ast_n$.
We write $f^\ast_{m,n}$ for $f^\ast_n\circ f^\ast_{n-1}\circ\dots\circ f^\ast_{m+1}\circ f^\ast_m$.
Then, we define $q_n^\ast\in E_n$ and $E_{n+1}\subseteq E\times\overline{L}^{n+1}$ as follows: 
\[q_{n}^\ast=f^\ast_{0,n}(q_{n}),\qquad E_{n+1}=S(E_n,K,q_{n}^\ast).\]
Let $p_{n+1}:E_{n+1}\to E_n$ be the projection.
Note that for $p_{0,n+1}=p_0\circ p_1\circ\dots\circ p_{n+1}:E_{n+1}\to E_0$, we have
\[p_{0,n+1}^{-1}(x)=
\begin{cases}
\{f_{0,n+1}^\ast(x)\},&\mbox{ if }x\not\in\{q_0,\dots,q_n\},\\
\{q_n^\ast\}\times K\times\{0\},&\mbox{ if }x=q_n,\\
f_{i+2,n+1}^\ast[\{q_i^\ast\}\times K\times\{0\}],&\mbox{ if }x=q_i\mbox{ for some $i<n$}.
\end{cases}
\]
In particular, the fibers $p_{0,n+1}^{-1}(q_0),\dots,p_{0,n+1}^{-1}(q_n)$ are copies of $K$, and the other fibers are singletons.

Then, we define $S(E,K,Q)$ as the following inverse limit:
\[S(E,K,Q)=\varprojlim(E_n,p_n).\]
We naturally identify $S(E,K,Q)$ with the following set:
\[\{(z_i)_{i\in\om}\in E\times\overline{L}^\om:(\forall n)\;\langle z_i:i\leq n\rangle=p_{n+1}(\langle z_i:i\leq n+1\rangle)\}.\]
Let $p:S(E,K,Q)\to E$ be the projection, that is, $p(z)=z_0$.
The following is the key property of the space $S(E,K,Q)$:

\begin{lemma}[{\cite[Lemma 2.6]{ChPo}}]\label{lem:main-ChPol}
Let $K$ and $L$ be at least two-dimensional continua whose Fr\'echet types are incomparable.
Then, no nonempty open subset of $S(\interval,L,\mathbb{Q})$ embeds in $S(\interval,K,\mathbb{Q})$.
\end{lemma}

By combining the Baire category argument with this lemma, we get the following:

\begin{lemma}\label{lem:weakly-1-gen-sep}
Assume that $X$ and $Y$ are at least two-dimensional computable continua whose Fr\'echet types are incomparable.
Let $x$ be a weakly $1$-generic real in $\interval$, and $x^\ast$ be the unique element in the $x$-th fiber in $S(\interval,X,\mathbb{Q})$.
Then, $x^\ast$ is not $tt$-equivalent to a point in $S(\interval,Y,\mathbb{Q})$.
\end{lemma}

\begin{proof}
Let $P$ be a $\Pi^0_1$ subset of $S=S(\interval,X,\mathbb{Q})$ such that $x^\ast\in P$.
Put $U=S\setminus P$.
We claim that $x^\ast\not\in p^{-1}[p[U]]$.
Note that if $t\in p[U]$, then $U$ intersects with $p^{-1}(t)$.
Thus, if moreover $t\not\in\mathbb{Q}$, then $p^{-1}(t)$ is a singleton; hence $p^{-1}(t)\subseteq U$.
However, since $x\not\in\mathbb{Q}$, $p^{-1}(x)=\{x^\ast\}$, and $x^\ast\not\in U$, we have $x\not\in p[U]$, that is, $x^\ast\not\in p^{-1}[p[U]]$.
Now, $Q:=S\setminus p^{-1}[p[U]]\subseteq P$ is also a $\Pi^0_1$ set which contains $x^\ast$.
Note that $Q$ is of the form $p^{-1}[\interval\setminus V]$ for some c.e.\ open set $V\subseteq\interval$.
Then, $\interval\setminus V$ has a nonempty interior since $\interval\setminus V$ is a $\Pi^0_1$ set containing a weakly $1$-generic real $x$.
Let $G\subseteq \interval\setminus V$ be a nonempty set which is open in $\interval$.
Then, $Q$ contains $p^{-1}[G]$, which is open in $S$.
By Lemma \ref{lem:main-ChPol}, $p^{-1}[G]$ cannot be embedded into $S(\interval,Y,\mathbb{Q})$ since $X$ and $Y$ have incomparable Fr\'echet types.
Therefore, $P$ does not embed into $S(\interval,Y,\mathbb{Q})$, by $p^{-1}[G]\subseteq Q\subseteq P$.
This implies that $x^\ast$ is not $tt$-equivalent to a point in $S(\interval,Y,\mathbb{Q})$ by Lemma \ref{lem:main-lemma}.
\end{proof}

\subsubsection{Dimension of $S(E,K,t)$.}
We hope to know the dimension-theoretic property of $S(E,K,t)$.
For instance, whenever $E$ and $K$ are $n$-dimensional, is $S(E,K,t)$ also $n$-dimensional?
To study the dimension-theoretic property of $S(E,K,t)$, we introduce an auxiliary notion.
We use $B_\ep(x)$ to denote the $\ep$-ball centered by $x$, and by $B_{\leq \ep}(x)$ and $B_{=\ep}(x)$ we mean its formal closure and formal boundary, that is, all points $y$ with $d(x,y)\leq \ep$ and $d(x,y)=\ep$, respectively.
We say that a metric $d$ on $S$ is {\em $n$-good} if the following condition holds:
\begin{itemize}
\item For any $x\in S$, there is a positive real $\ep>0$ such that for any $y\in S$ and $\delta>0$, if $d(x,y)\not=\delta$, then $B_\ep(x)\cap B_{=\delta}(y)$ is at most $(n-1)$-dimensional.
\end{itemize}
This condition clearly implies that the small induction dimension of $S$ is at most $n$; hence $\dim(S)\leq n$.
For instance, the usual Euclidean distance on $\mathbb{R}^n$ is an $n$-good metric.

\begin{lemma}\label{lem:n-good}
Suppose that $d_H$ is a metric on $\interval^\om$ such that $d_K:=d_H\upto K$ is an $n$-good metric on $K$.
If a space $E$ admits an $n$-good metric, then so does $S(E,K,t)$.
\end{lemma}

\begin{proof}
Let $d_E$ be an $n$-good metric of $E$.
Put $S=S(E,K,t)$.
Recall that $S\subseteq E\times \interval^\om\times\interval$.
Define the metric $d_S$ on $S$ by the sup metric, that is, 
\[d_S((x_0,x_1,x_2),(y_0,y_1,y_2))=\max\{d_E(x_0,y_0),d_H(x_1,y_1),|x_2-y_2|\}.\]
Fix $z=(z_0,z_1)\in S\subseteq E\times\overline{L}$.
We will define $\ep>0$ such that $B_\ep(z)\cap B_{=\delta}(y)$ is at most $(n-1)$-dimensional for any point $y\in S$ and positive real $\delta\not=d_S(z,y)$.

We first consider the case $z_0\not=t$.
Then let $\varepsilon(0)$ be a positive rational such that $t\not\in B_{\varepsilon(0)}(z_0)$.
Since $z_0\not=t$ implies $z_1\in L$, for a sufficiently small $\ep(1)$, we have $B_{\ep(1)}(z_1)\subseteq L_k$ for some $k$.
Choose $\ep<\min\{\ep(0),\ep(1)\}$ such that $\ep$ witnesses that $d_E$ is $n$-good.
Fix $y=(y_0,y_1)\in S$ and $\delta>0$ such that $d_S(z,y)\not=\delta$.
We need to show that $B_\ep(z)\cap B_{=\delta}(y)$ is at most $(n-1)$-dimensional.
Note that
\[B_{=\delta}(y)=(B_{=\delta}(y_0)\times B_{\leq \delta}(y_1))\cup(B_{\leq\delta}(y_0)\times B_{=\delta}(y_1)).\]

For the former product, recall that, if $t\not\in B\subseteq E$, then $p^{-1}[B]=f^\ast[B]$ is homeomorphic to $B$.
Since $\ep<\ep(0)$, $B_\ep(z)\cap(B_{=\delta}(y_0)\times\overline{L})$ embeds into $B_\ep(z_0)\cap B_{=\delta}(y_0)$, which is at most $(n-1)$-dimensional by $n$-goodness of $d_E$.

For the latter one, our assumption $\ep<\ep(1)$ implies that
\[p[B_\ep(z)]\subseteq B_\ep(z_0)\cap f^{-1}[L_k].\]
Moreover, we always have $B_{\leq\delta}(y_0)\times B_{=\delta}(y_1)\subseteq p^{-1}f^{-1}[B_{=\delta}(y_1)]\cup p^{-1}(t)$, but our assumption $\ep<\ep(0)$ implies that $B_\ep(z)\cap p^{-1}(t)$ is empty; hence
\begin{align}\label{equation:a}
B_\ep\cap(B_{\leq\delta}(y_0)\times B_{=\delta}(y_1))\subseteq p^{-1}f^{-1}[B_{=\delta}(y_1)\cap L_k]=ff^{-1}[B_{=\delta}(y_1)\cap L_k].
\end{align}
Now, $B_{=\delta}(y_1)\cap L_k$ has at most two points, since $B_{=\delta}(y_1)$ is a sphere, and $L_k$ is a line segment.
Thus, there are at most two reals $p,q$ such that
\begin{align}\label{equation:b}
f^{-1}[B_{=\delta}(y_1)\cap L_k]\subseteq B_{=p}(t)\cup B_{=q}(t).
\end{align}
By (\ref{equation:a}) and (\ref{equation:b}), $B_\ep(z)\cap(B_{\leq \delta}(y_0)\times B_{=\delta}(y_1))$ embeds into $B_\ep(z_0)\cap (B_{=p}(t)\cup B_{=q}(t))$, which is at most $(n-1)$-dimensional by $n$-goodness of $d_E$.

Consequently, $B_\ep(z)\cap B_{=\delta}(y)$ is the union of two closed sets each of which is at most $(n-1)$-dimensional.
Thus, by the sum theorem, $B_\ep(z)\cap B_{=\delta}(y)$ is also at most $(n-1)$-dimensional.

Next consider the case $z_0=t$.
Then let $\ep$ be a witness of $n$-goodness of both $d_E$ and $d_K$.
Fix $y\in S$ and $\delta>0$ such that $\delta\not=d(z,y)$.
First consider $B_0=B_\ep(z)\cap(B_{=\delta}(y_0)\times B_{\leq\delta}(y_1))$.
Since $z_0=t\not\in B_{=\delta}(y)$, if $r=(r_0,r_1)\in B_0$ then $r_0\not=t$.
This means that $B_0$ can be embedded into $B_\ep(z_0)\cap B_{=\delta}(y_0)$ as before, which is at most $(n-1)$-dimensional.

Then define $B_1=B_\ep(z)\cap(B_{\leq\delta}(y_0)\times B_{=\delta}(y_1))$.
Note that $B_{=\delta}(y_1)$ is the disjoint union of $C_0=B_{=\delta}(y_1)\cap L$ and $C_1=B_{=\delta}(y_1)\cap (K\times \{0\})$.
Again, for any $k$, $B_{=\delta}(y_1)\cap L_k$ has at most two points.
Thus, there are at most countably many $(r(k))_{k\in\om}$ such that
\[f^{-1}[B_{=\delta}(y_1)\cap L]\subseteq \bigcup_kB_{=r(k)}(t).\]
Note that $B_\ep(z_0)\cap B_{=r(k)}(t)$ is at most $(n-1)$-dimensional by $n$-goodness, and homeomorphic to $B_\ep(z)\cap p^{-1}[B_{=r(k)}(t)]$.
Now $B_\ep(z)\cap p^{-1}[B_{=r(k)}(t)]$ is an at most $(n-1)$-dimensional closed subset of $B_\ep(z)$, and we have
\[B_\ep(z)\cap C_0=B_\ep(z)\cap\bigcup_kp^{-1}[B_{=r(k)}(t)].\]
This concludes that $B_\ep(z)\cap C_0$ is a countable union of closed subsets each of which is at most $(n-1)$-dimensional.

Now $B_\ep(z_1)\cap C_1$ is at most $(n-1)$-dimensional by $n$-goodness, and 
\[S\cap (E\times K\times \{0\})\subseteq p^{-1}(t).\]
This means that for any $C\subseteq K\times\{0\}$, $S\cap (E\times C)$ is homeomorphic to $C$.
Thus, $B_\ep(z)\cap(E\times C_1)$ can be embedded into $C_1$, which is at most $(n-1)$-dimensional.
Consequently, $B_\ep(z)\cap B_{=\delta}(y)$ is the union of countably many closed subsets each of which is at most $(n-1)$-dimensional.
By the sum theorem, this concludes that $B_\ep(z)\cap B_{=\delta}(y)$ is at most $(n-1)$-dimensional.
\end{proof}

\begin{lemma}\label{lem:n-dim-com-cont}
If $K$ is a computable continuum, so is $S(\interval,K,\repsp{Q})$.
Moreover, if $K$ is $n$-dimensional, and $d_H\upto K$ is an $n$-good metric, then $S(\interval,K,\repsp{Q})$ is also $n$-dimensional.
\end{lemma}

\begin{proof}
We first show that $S(\interval,K,\repsp{Q})$ admits a computable metrization.
Note that we have constructed each $S(\interval,K,t)$ as a subset of $\interval\times\interval^\om\times\interval$, which is effectively homeomorphic to $\interval^\om$.
Thus, we think of $S(\interval,K,\repsp{Q})$ as a subset of $\interval^\om$ is a straightforward manner, and let $d$ be a computable metric on $\interval^\om$.
Let $(\tilde{q}_e)_{e\in\om}$ be a computable dense subset of $\interval\setminus \repsp{Q}$, and then define $a_e\in S(\interval,K,\repsp{Q})$ as the unique element in the fiber $p^{-1}(\tilde{q}_e)$.
Clearly, $\{a_e:e\in\om\}$ is a dense subset of $S(\interval,K,\repsp{Q})$.
Note that we can get $a_e$ by applying computable functions $(f^\ast_n)_{n\in\om}$ to the computable point $\tilde{q}_e$, and therefore, $a_e$ is also computable uniformly in $e$.
Thus, $(d,e)\mapsto d(a_d,a_e)$ is computable.
This shows that $S(\interval,K,\repsp{Q})$ is computably metrizable via the induced metric.

We show that $S=S(\interval,K,\repsp{Q})$ is $n$-dimensional.
Clearly $\dim(S)\geq n$ since $K$ embeds into $S$.
By Lemma \ref{lem:n-good}, one can inductively see that each $E_s$ is metrized by an $n$-good metric.
In particular, $\dim(E_s)\leq n$ for any $s\in\om$.
Thus, $S$ is $n$-dimensional since $S$ is an inverse limit of a sequence of $n$-dimensional compacta (cf.\ Engelking \cite[Theorem 1.13.4]{EngBook}).
\end{proof}

\subsubsection{Independent Fr\'echet types}
Let $\mathbb{S}^n$ be the $n$-sphere, fix two homeomorphic disjoint closed neighborhoods $A,B\in\mathbb{S}^n$, and let $h$ be a homeomorphism between $A$ and $B$.
Then, let $K^n_m$ be an union of $m$ spheres obtained by gluing the closed neighborhood $A$ in the $i$-th sphere to the closed neighborhood $B$ in the $(i+1)$-th sphere, that is,
\[K^n_m=\{(i,x):i<m\mbox{ and }x\in\mathbb{S}^n\}/\sim,\]
where $(i,a)\sim(i+1,h(a))$ for any $a\in A$ and $i<m-1$.
If $n>1$, then it is clear that $K^n_m$ is an $n$-dimensional computable continuum which is not disconnected by a point, and moreover, if $k>m$ then $K^n_k$ cannot be embedded into $K^n_m$.

Given $(K^n_m)_{m\in\om}$, one can get a countable collection of $n$-dimensional computable continua with pairwise incomparable Fr\'echet types as in \cite[Lemma 3.2]{ChPo}.
For completeness, we here present an explicit construction.
Let $a_m$ ($b_m$, resp.)\ be a point in the first (last, resp.)\ $n$-sphere in $K^n_m$.
Fix a sufficiently fast-growing computable function $\kappa:\om\to\om$.
For any increasing function $g:\om\to\om$, connect $\kappa\circ g(0)$ many $K^n_{g(0)}$'s by identifying $a_{g(0)}$ in the $i$-th $K^n_{g(0)}$ with $b_{g(0)}$ in the $(i+1)$-th $K^n_{g(0)}$.
Similarly, connect $\kappa\circ g(1)$ many $K^n_{g(1)}$ in a similar manner.
Then, we connect these two chains by identifying $a_{g(0)}$ in the last link $K^n_{g(0)}$ of the first chain with $b_{g(1)}$ in the first link $K^n_{g(1)}$ of the second chain.
Continue this procedure.
We eventually connect infinitely many chains, and then consider the one-point compactification.
We write $K^n_g$ for the resulting continuum.

If $f$ and $g$ are almost disjoint, $K^n_f$ and $K^n_g$ have incomparable Fr\'echet types (see \cite[Lemma 3.2]{ChPo} for the detail).
If $g$ is computable, it is clear that $K^n_g$ is also computable.
Note that $K^n_g$ is $n$-dimensional, and admits a computable $n$-good metric; to see this, for instance, embed the $n$-sphere into $\mathbb{R}^{n+1}$ as the surface of an $(n+1)$-dimensional hypercube, and then consider the standard Euclidean metric.

\begin{cor}\label{cor:incomparable-degrees}
For any $n>1$, there is a collection $(\repsp{S}_i)_{i\in\om}$ of $n$-dimensional computable continua satisfying the following:
For any weakly $1$-generic real $x\in\interval$, there are points $y_i\in\repsp{S}_i$, $i\in\om$, such that 
\[\pt{x}{\interval}<_{tt}\pt{y_i}{\repsp{S}_i}\equiv_T\pt{x}{\interval},\]
and $\pt{y_i}{\repsp{S}_i}$ is not $tt$-equivalent to a point in $\repsp{S}_j$ for any $j\not=i$.
\end{cor}

\begin{proof}
Define $g_i(n)={\langle i,n\rangle}$, and put $\repsp{S}_i=S(\interval,K^n_{g_i},\mathbb{Q})$.
By Lemma \ref{lem:n-dim-com-cont}, $\repsp{S}_i$ is an $n$-dimensional computable continuum.
Then, let $y_i$ be the unique element in the $x$-th fiber of $\repsp{S}_i$.
We have $x\leq_{tt}y_i$ since the projection is total and $p(y_i)=x$.
We also have $y_i\leq_{T}x$ since we can effectively get $y_i$  from $x$ by iterating $f^\ast_n$ (which are defined on irrationals).
For the latter part, the ranges of $g_i$ and $g_j$ are disjoint whenever $i\not=j$.
Thus, by Lemma \ref{lem:weakly-1-gen-sep}, $\pt{y_i}{\repsp{S}_i}$ is not $tt$-equivalent to a point in $\repsp{S}_j$ as desired.
\end{proof}

\begin{proof}[Proof of Theorem \ref{thm:usl-embedding2}]
It suffices to construct a usl-embedding of the free countable Boolean algebra into $\mathfrak{D}_{tt}^\mathbf{d}(n\text{-}\mathrm{dim},\mathrm{cont})$.
Let $(\repsp{S}_i)_{i\in\om}$ be the family in Corollary \ref{cor:incomparable-degrees}, and define $\repsp{S}_A$ as the topological sum $\bigsqcup_{i\in A}\mathcal{S}_i$ for any $A\subseteq\om$.
Then, $A\mapsto\repsp{S}_A$ is clearly a usl-embedding of the Boolean algebra of computable subsets of $\om$ into $\mathfrak{D}_{tt}^\mathbf{d}(n\text{-}\mathrm{dim},\mathrm{cont})$.
This concludes the proof since the Boolean algebra of computable subsets of $\om$ includes the free countable Boolean algebra, which also includes any countable upper semilattice.
\end{proof}

\section{Open Question}

\begin{question}
Does there exist infinitely many $tt$-degree structures of computable arc-like continua?
\end{question}

\begin{question}
Does there exists a pair of computable metric spaces which are first-level Borel isomorphic, but have different $tt$-degree structures?
\end{question}


\subsection*{Acknowledgements}
The author was partially supported by JSPS KAKENHI Grant 17H06738 and 15H03634.
The author also thanks JSPS Core-to-Core Program (A. Advanced Research Networks) for supporting the research.
The author would like to thank Arno Pauly for valuable discussions.

\bibliographystyle{plain}
\bibliography{first-level.bib}

\end{document}